\documentclass[12pt,reqno]{amsart}

\usepackage[l2tabu, orthodox]{nag}
\usepackage[T1]{fontenc}
\usepackage[utf8]{inputenc}
\usepackage[english]{babel}
\usepackage{csquotes}

\usepackage{microtype}
\usepackage{geometry}
\usepackage[stable]{footmisc}

\usepackage[%
  backend=biber,hyperref,
  bibencoding=auto,style=alphabetic,citestyle=alphabetic,
  doi=true,url=true,isbn=false,
  maxbibnames=99,giveninits=true
  ]{biblatex}

\AtEveryBibitem{\clearlist{language}}
\DeclareFieldFormat[book]{pages}{}
\renewbibmacro{in:}{%
  \ifentrytype{article}{}{\printtext{\bibstring{in}\intitlepunct}}}
\usepackage{amsmath}
\usepackage{mathtools}
\usepackage{fixmath}
\usepackage{amsthm}
\usepackage{amssymb}
\usepackage{tikz-cd}

\usepackage[unicode]{hyperref}

\usepackage{enumitem}
\setlist[enumerate]{nosep,label=\roman*),font=\normalfont}
\makeatletter
\let\savedenumerate=\enumerate%
\let\savedendenumerate=\endenumerate%
\renewenvironment{enumerate}{~\savedenumerate}{\savedendenumerate}
\makeatother

\usepackage{color}

\usepackage{chngcntr}

\usepackage{float}
\usepackage{caption}
\usepackage{subcaption}

\usepackage{amsmath}
\usepackage{amsthm}
\usepackage{amssymb}
\usepackage{mathrsfs}

\theoremstyle{plain}
\newtheorem{theorem}[equation]{Theorem}
\newtheorem*{theorem*}{Theorem}

\newtheorem{conjecture}[equation]{Conjecture}
\newtheorem{corollary}[equation]{Corollary}
\newtheorem*{corollary*}{Corollary}
\newtheorem{lemma}[equation]{Lemma}
\newtheorem{proposition}[equation]{Proposition}
\newtheorem*{proposition*}{Proposition}
\newtheorem*{lemma*}{Lemma}

\theoremstyle{remark}
\newtheorem{remark}[equation]{Remark}

\theoremstyle{definition}
\newtheorem{definition}[equation]{Definition}
\newtheorem{notation}[equation]{Notation}
\newtheorem*{notation*}{Notation}
\newtheorem{example}[equation]{Example}
\newtheorem{convention}[equation]{Convention}
\newtheorem{situation}[equation]{Situation}
\newtheorem*{situation*}{Situation}
\newtheorem{construction}[equation]{Construction}

\numberwithin{equation}{section}

	{\begin{proof}[Answer]\renewcommand{\qedsymbol}{}}%
	{\end{proof}}

\newcounter{char}
\loop\ifnum\value{char}<27
  \expandafter\edef\csname\Alph{char}bb\endcsname{\noexpand\mathbb{\Alph{char}}}
  \expandafter\edef\csname\Alph{char}ca\endcsname{\noexpand\mathcal{\Alph{char}}}
  \expandafter\edef\csname\Alph{char}sc\endcsname{\noexpand\mathscr{\Alph{char}}}
  \expandafter\edef\csname\Alph{char}fr\endcsname{\noexpand\mathfrak{\Alph{char}}}
  \expandafter\edef\csname\alph{char}fr\endcsname{\noexpand\mathfrak{\alph{char}}}
  \expandafter\edef\csname\Alph{char}bf\endcsname{\noexpand\mathbf{\Alph{char}}}
  \expandafter\edef\csname\alph{char}bf\endcsname{\noexpand\mathbf{\alph{char}}}
  \ifnum\value{char}=1\else
    \expandafter\edef\csname\Alph{char}\Alph{char}\endcsname{\noexpand\mathbb{\Alph{char}}}
  \fi
\addtocounter{char}{1}
\repeat

\let\oldAA\AA
\renewcommand*{\AA}{\ifmmode\mathbb{A}\else\oldAA\fi}

\DeclareMathOperator{\Emb}{Emb}
\DeclareMathOperator{\Fun}{Fun}
\DeclareMathOperator{\ho}{h}
\DeclareMathOperator{\Ho}{H}

\DeclareMathOperator*{\holim}{holim}
\DeclareMathOperator{\Hom}{Hom}
\DeclareMathOperator{\im}{im}
\DeclareMathOperator{\kernel}{ker}
\DeclareMathOperator{\Map}{Map}
\DeclareMathOperator{\Mor}{Mor}
\DeclareMathOperator{\Nv}{N}
\DeclareMathOperator{\Ob}{Ob}
\DeclareMathOperator{\Tot}{Tot}

\newcommand*{\Category}[1]{\mathrm{\mathbf{#1}}}

\newcommand*{\CGH}{\mathcal{C}\mathcal{G}\mathcal{H}}

\newcommand*{\CatDisc}[1]{\Category{Disc}_{\mathrm{#1}}}
  \newcommand*{\infCatDisc}[1]{\mathcal{D}\mathrm{isc}_{\mathrm{#1}}}
  \newcommand*{\CatDiscOver}[2]{\Category{Disc}_{\mathrm{#1} / \mathrm{#2}}}
  \newcommand*{\infCatDiscOver}[2]{\mathcal{D}\mathrm{isc}_{\mathrm{#1} / \mathrm{#2}}}

\newcommand*{\infFun}{\mathcal{F}un}

\newcommand*{\infHType}{\mathcal{H}o}

\newcommand*{\CatsInd}{\pmb{\mathrm{\Delta}}}

\newcommand*{\CatMan}[1]{\Category{Man}_{\mathrm{#1}}}
\newcommand*{\infCatMan}[1]{\mathcal{M}\mathrm{an}_{\mathrm{#1}}}
\newcommand*{\CatManOver}[2]{\Category{Man}_{\mathrm{#1} / \mathrm{#2}}}
\newcommand*{\infCatManOver}[2]{\mathcal{M}\mathrm{an}_{\mathrm{#1} / \mathrm{#2}}}
   
\newcommand*{\CatPow}[1]{\Category{Pow}(#1)}
  
\newcommand*{\Open}[2]{\Category{Open}^{#2}\left(#1 \right)}
  \newcommand*{\bOpen}[2]{\Category{Open}_{\partial}^{#2}\left(#1 \right)}
  \newcommand*{\fbOpen}[1]{\Category{Open}_{\partial}\left(#1 \right)_{\mathrm{fin}}}

\newcommand*{\quot}[2]{#1 / \mathord #2}
\newcommand*\restr[2]{\left.#1\right|_{#2}}
\newcommand*{\id}{\operatorname{id}}
\newcommand*{\pr}{\operatorname{pr}}

\newcommand*{\blank}{{-}}

\newcommand*{\E}{\mathrm{E}}

\newcommand*{\Conf}{\mathrm{Conf}}
\newcommand*{\cinfty}{\mathrm{C}^{\infty}}
\newcommand*{\Embcos}{\mathfrak{E}mb}
\newcommand*{\udisc}{\mathrm{D}}
\newcommand*{\ssimplex}{\mathrm{\Delta}}
\newcommand*{\I}{\mathrm{I}}

\newcommand*{\usphere}{\mathrm{S}}
\newcommand*{\op}{\mathrm{op}}
\newcommand*{\pT}{\mathrm{T}}

\usepackage{xspace}

\newcommand*{\eg}{e.g.\@\xspace}
\newcommand*{\ie}{i.e.\@\xspace}
\newcommand*{\cf}{cf.\@\xspace}

\newcommand*{\rom}[1]{\textrm{\romannumeral #1)}}

\hypersetup{colorlinks=true, linktocpage}
\hypersetup{allcolors=blue}
\hypersetup{linkcolor=blue, citecolor=blue, urlcolor=magenta}

\title[Goodwillie's cosimplicial model for knots and its applications]{Goodwillie's cosimplicial model for the space of long knots and its applications}

\author{Yuqing Shi}
\thanks{Email: y.shi@uu.nl. Affiliation: Utrecht University, Utrecht, the Netherlands}

\date{\today}

\hypersetup{
  pdfauthor={Yuqing Shi},
  pdftitle={Goodwillie's cosimplicial model for knots and its applications},
  pdfcreator=pdflatex
}

\begin{document}


\begin{abstract}
  We work out the details of a correspondence observed by Goodwillie between cosimplicial spaces and good functors from a category of open subsets of the interval to the category of compactly generated weak Hausdorff spaces. 
  Using this, we compute the first page of the integral Bousfield--Kan homotopy spectral sequence of the tower of fibrations, given by the Taylor tower of the embedding functor associated to the space of long knots.
  Based on the methods of~Conant, we give a combinatorial interpretation of the differentials $d^1$ mapping into the diagonal terms, by introducing the notion of $(i, n)$-marked unitrivalent graphs.
\end{abstract}

\maketitle

\tableofcontents


\section*{Introduction}
\counterwithout{equation}{subsection}

Manifold calculus is a theory introduced by Goodwillie and Weiss, \cf \cite{Wei99} and \cite{GKW01}, that produces a sequence of functors approximating a given good functor (Definition \ref{def:goofun}) from the category of open subsets of a manifold to the category $\CGH$ of compactly generated weak Hausdorff spaces.
Let $M$ and $N$ be smooth manifolds. 
We are interested in studying the space $\Emb_{\partial}(M, N, f)$ of smooth embeddings that are germ equivalent on the boundary of $M$ to a fixed smooth boundary-preserving embedding~$f \colon M \to N$.
We topologise this space with the compact-open topology.
Now we can apply manifold calculus to the embedding functor~$\Emb_{\partial}(\blank, N, f) \colon \Open{M}{}^{\op}_{\partial} \to \CGH$, sending an open subset $V \subseteq M$ with $\partial M \subseteq V$ to the embedding space $\Emb_{\partial}(V, N, f)$.
In~this way we obtain information about the space $\Emb_{\partial}(M, N, f)$ by studying the embedding functor and its sequence of approximations.

In this paper, we focus on analysing the following embedding functor
\begin{align*}
	 \Emb(\blank) \colon \bOpen{\I}{}^{\op} &\to \CGH \\ 
	 V &\mapsto \Emb_{\partial}(V, \RR^2 \times \udisc^1, c)
\end{align*}
associated to the space $\Kca = \Emb_{\partial}(\I, \RR^2 \times \udisc^1, c)$ of long knots where $c$ is an embedding representing the trivial long knot.  
Our motivation for studying this embedding functor is its close relation to knot theory, in particular, the theory of Vassiliev invariants.
In Section~\ref{sec:manicalknot} we recall some background on manifold calculus and give a precise definition of the space $\Kca$.
At the end of this section, we summarise briefly the original construction of the Vassiliev invariants and present how the theory of these knot invariants and the theory manifold calculus relate. 

As mentioned above, manifold calculus associates to $\Emb(\blank)$ a sequence of polynomial functors $\pT_n \Emb(\blank)$ (Definition \ref{def:polyfun}) approximating the functor $\Emb(\blank)$, 
\[
  \begin{tikzcd}
    &  \Emb(\blank) \arrow[dd, "\eta_n"] \arrow[rdd, "\eta_{n-1}"]  \arrow[rrrdd, "\eta_0"]  \arrow[ldd]&  &  & &\\ \\
      \dots \arrow[r, "r_{n+1}"'] & \pT_n \Emb(\blank) \arrow[r, "r_{n}"'] & \pT_{n-1} \Emb(\blank) \arrow[r, "r_{n-1}"'] & \dots  \arrow[r, "r_1"'] & \pT_0 \Emb(\blank).
  \end{tikzcd}
\]
The induced map $\pi_{0}\left(\eta_n(\I)\right) \colon \pi_0\left(\Emb(\I)\right) \to \pi_{0}\left(\pT_n \Emb(\I)\right)$ is an additive Vassiliev invariant of degree at most $n-1$ and conjecturally it is the universal one, \cf \cite{BCKS17}.
It is known that the map $\pi_0(\eta_{n}(\I))$ is surjective, \cf \cite{Kos20}.

Our approach to understanding the maps $\pi_{0}\left(\eta_n(\I)\right)$ is to compute the Bousfield--Kan homotopy spectral sequence associated to the tower of fibrations
\begin{equation}\label{eq:knotfibration}
   \cdots \xrightarrow{r_{n+1}(\I)} \pT_n \Emb(\I) \xrightarrow{r_{n}(\I)} \pT_{n-1} \Emb(\I) \xrightarrow{r_{n-1}(\I)} \cdots \xrightarrow{r_{1}(\I)} \pT_0 \Emb(\I). 
\end{equation}

As a first step, we study in Section~\ref{sec:cosm} the construction \cite[Contruction 5.1.1]{GKW01} that relates cosimplicial spaces with good functors from $\bOpen{\I}{}$ to the category of spaces.
In Theorem~\ref{thm:goodwilliecos} we give a more precise formulation of this construction, summarised as the theorem below.

\begin{theorem}
  There is an isomorphism between the homotopy class of augmented cosimplicial spaces and the homotopy class of good functors $F \colon \bOpen{\I}{}^{\op} \to \CGH$.
\end{theorem}


We were not able to find a proof of the above theorem or of \cite[Contruction~5.1.1]{GKW01} in the literature. 
Thus, we give a proof in Section~\ref{sec:cosm}, using some results of \cite{AF15}.
Therefore, we can associate a cosimplicial space $\Embcos^{\bullet}$ to the embedding functor $\Emb(\blank)$ via the above equivalence, which provides us with a cosimplicial model for the tower of fibrations~(\ref{eq:knotfibration}), namely
\begin{align*}
  \Tot^{n}\Embcos^{\bullet} &\simeq \pT_n \Emb(\I)\\
  \Tot \Embcos^{\bullet} &\simeq \holim_{n \geq 0} \pT_n \Emb(\I).
\end{align*}
We find the construction of this cosimplicial model very natural, yet it is the least used one in the study of the space of long knots. 
There is another cosimplicial model $C^{\bullet}$ of the tower of fibration~(\ref{eq:knotfibration}) which is defined using the Fulton--MacPherson compactification, \cf \cite{Sin06} and \cite{BCKS17}. 
In Remark~\ref{rmk:othermodel} we briefly recall the construction and properties of~$C^{\bullet}$.

The $n$-th level $\Embcos^{n}$ of the cosimplicial space $\Embcos^{\bullet}$ is weakly homotopy equivalent to the cartesian product of the configuration space $\Conf_n(\RR^2 \times \udisc^1)$ and $n$ copies of the sphere $\usphere^2$.
In Section~\ref{sec:hssknotcal}, we compute the Bousfield--Kan homotopy spectral sequence $\{E_{p,q}^r\}_{p, q \geq 0}$ with integral coefficients associated to $\Embcos^{\bullet}$, using a presentation of the homotopy groups of the configuration spaces via iterated Whitehead products.
More specifically, we are able to give an explicit description of the abelian groups $E_{p, q}^1$ (Proposition~\ref{prop:e1page} and Remark~\ref{rmk:e1page}).
Thanks to this, we obtain a simplification of the calculation of the differentials mapping into the diagonal terms on the $E^1$-page.

\begin{theorem}[Theorem~\ref{prop:d1simplified}]
  For $p \geq 3$,\footnote{For $p = 1$, the differential is null because the domain and the codomain are null. For $p = 2$, the differential is an isomorphism by concrete calculation.} the differential
  $
  d^1 \colon E_{p-1, p}^1 \to E_{p, p}^1
  $
  can be expressed as an explicit sum of four iterated Whitehead products.
\end{theorem}

In Section~\ref{sec:combidi}, we give a combinatorial interpretation of the groups $E_{p, p}^1$ and~$E_{p-1, p}^1$, as well as the differentials $d^1 \colon E_{p, p}^1 \to E_{p-1, p}^1$, based on the methods of \cite{Con08}.
The group~$E_{p, p}^1$ is isomorphic to the abelian group $\Tca_{p-1}$ of labelled unitrivalent trees of degree~$p-1$ with a total ordering on its leaves, modulo AS- and IHX-relations  (Proposition~\ref{prop:eppcomb}).
In Figure~\ref{fig:intro-feynmantree} we draw an example of a labelled unitrivalent tree of degree $4$.

\begin{figure}[!ht]
    \includegraphics[width=\textwidth]{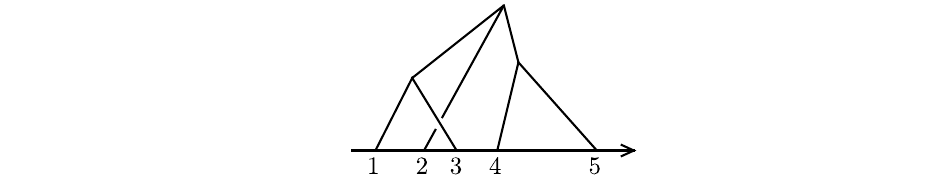}
    \caption{A labelled unitrivalent tree of degree $4$. The arrow is not part of the tree.}
    \label{fig:intro-feynmantree}
\end{figure}

In order to describe the groups $E_{p-1, p}^1$, we introduce the notion of $(i, p)$-marked unitrivalent graphs (Definition~\ref{def:ijmarkedFdiag}).
In Figure~\ref{fig:intro-feynmanloop} we draw an example of a marked unitrivalent graph of degree $8$.
\begin{figure}[!ht]
    \includegraphics[width=\textwidth]{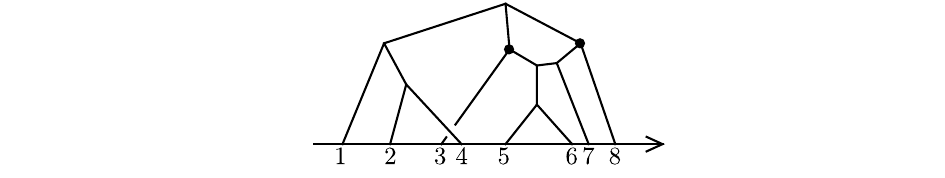}
    \caption{A $(3,8)$-marked unitrivalent graph. The black dots indicate the marked nodes. The arrow is not part of the graph.}
    \label{fig:intro-feynmanloop}
\end{figure}

\begin{proposition}[Proposition~\ref{prop:ep-1pcomb}]
  Let $\Dca_{p-1}$ be the abelian group generated by $(i, p-1)$-marked unitrivalent graphs with $1 \leq i \leq p$, modulo AS- and IHX\textsuperscript{sep}-relations. 
  Then we can identify $\Dca_{p-1}$ with the torsion-free part of $E_{p-1,p}^1$.
\end{proposition}

The differentials $d^1$ have the following interpretation using unitrivalent graphs.

\begin{theorem}[Theorem~\ref{thm:d1comb}]
  Let $p \geq 4$.
  Under the identifications above, the differential $d^1 \colon E_{p-1, p} \to E_{p, p}$ maps a $(k,p-1)$-marked unitrivalent graph $\Gamma_{k, p-1}$ to the linear combination $\Gamma_{p}^{1} - \Gamma_{p}^{2} - (\Gamma_{k}^{2} - \Gamma_{k}^{1})$ of unitrivalent trees of degree $p-1$, where 
  \begin{enumerate}
    \item the linear combination $\Gamma_{p}^1 - \Gamma_p^2$ is obtained by performing the STU-relation (Definition \ref{def:stu}) on $\Gamma_{k, p-1}$ at the edge connecting the leaf labelled by $p-1$ and the marked node $v_{p-1}$, and
    \item the linear combination $\Gamma_{k}^2 - \Gamma_k^1$ is obtained by performing the STU-relation on $\Gamma_{k, p-1}$ at the edge connecting the leaf labelled by $k$ and the marked node $v_{p-1}$.
  \end{enumerate} 
\end{theorem}

\begin{figure}
\includegraphics[width=\textwidth]{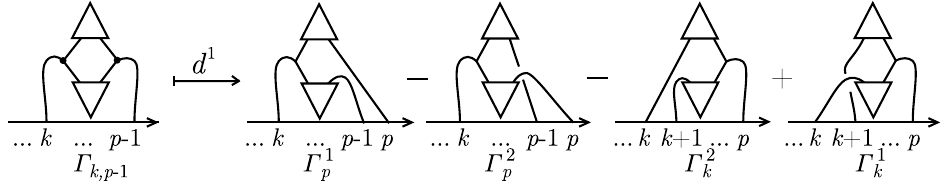}
\caption{An example of $d^1$ applied to a $(k,p-1)$-marked unitrivalent graph. The triangles are placeholders for subgraphs, which stay unmodified.}
\label{fig:intro-d1form}
\end{figure}

See Figure \ref{fig:intro-d1form} for a visualisation of the graphs occurring in the theorem.
We call the equivalence relation on the abelian group $\Tca_{p-1}$ (labelled unitrivalent trees of degree $p-1$) generated by the image of $d^1$ the STU\textsuperscript{2}-equivalence relation (Definition~\ref{def:stu}).
As a corollary we have the following proposition.

\begin{proposition}[Conant]\label{prop:intro-d1comb}
 A tree $\tau \in \Tca_{p-1}$ is STU\textsuperscript{2}-equivalent to 0 if and only if $\tau \in \im(d^1)$ under the isomorphism $E_{p,p}^1 \cong \Tca_{p-1}$ of Proposition~\ref{prop:eppcomb}.
\end{proposition}

We also obtain a combinatorial interpretation of $E_{p,p}^2$ for $p \geq 1$.

\begin{corollary}[Corollary~\ref{cor:epp2}]
  \begin{enumerate}
    \item $E_{p,p}^2$ is isomorphic to the abelian group generated by unitrivalent tree of degree $p-1$, modulo AS-, IHX-, and STU\textsuperscript{2}-relations, for $p \geq 4$.
    \item $E_{3,3}^2 \cong E_{3,3}^1 \cong \Tca_{2} \cong \ZZ$, for $p = 3$. 
    \item $E_{p,p}^2 = 0$, for $p = 0, 1, 2$.
  \end{enumerate}  
\end{corollary}

We conclude in Section~\ref{sec:furtherwork} by illustrating the connection between the manifold calculus tower of $\Emb(\blank)$ and some geometrical and combinatorial aspects of Vassiliev knot invariants, and mention some future work.

\subsection*{Acknowledgements}
Most of the results in this paper are part of my master thesis, written under the supervision of Peter Teichner.
To him belongs my gratitude for his suggestions, guidance and support. 
I want to thank Lukas Brantner, Danica Kosanovic, Achim Krause for the helpful conversations during the thesis project.
I would like to thank Jack Davies, Tobias Dyckerhoff, Gijs Heuts, Geoffroy Horel, Pablo Magni and Mingcong Zeng for their feedback on the current version and for their encouragement.

\counterwithin{equation}{section}

\begin{notation*}
	Throughout the text, we denote by $\I$, $\udisc^n$ and $\usphere^n$ the unit interval, the unit~$n$-disk and the unit $n$-sphere, respectively.
\end{notation*}

\begin{situation*}
	We work with simplicial categories, \ie categories enriched over the category of simplicial sets equipped with the Kan model structure\footnote{See \cite[Appendix A.2.7]{Lur09} for the Kan model structure on the category of simplicial sets.}.
	Most of the categories we consider in this text are by nature enriched over the category of compactly generated weakly Hausdorff spaces.	
	We regard them as simplicial categories by applying the singular complex functor to the mapping spaces.
	For simplicity, we call the mapping simplicial sets obtained as above again \emph{mapping spaces}.
	In particular, every ordinary category is a topologically enriched category with discrete morphism spaces.
	Thus we also consider them as simplicial categories.
	For an introduction on (simplicial) enriched categories and enriched functors see \cite[A.1.4, A.3]{Lur09} and \cite[Chapter 9]{Hir03}.
	
	We denote by $\holim$ the \emph{homotpy limits} in simplicial categories.
	We refer the reader to \cite[Chapter 18]{Hir03} for a detailed explanation of homotopy limits.
	For the calculation of homotopy limits in the category of compactly generated weak Hausdorff spaces see \cite[Chapter 18.2]{Hir03}, \cite[Chapter 6]{Rie14} and \cite[Section 8]{MV15}. 
\end{situation*}

\section{Background and motivation}\label{sec:manicalknot}


\subsection{Manifold Calculus}\label{sec:manical}
\counterwithin{equation}{subsection}

In this section we are going to briefly introduce the basic building blocks of manifold calculus.
The main references for this section are \cite{BW13}, \cite{Wei99} and \cite{GW99}, which also contain further motivation for this subject.

\begin{definition}\label{def:cats}
  \begin{enumerate}
    \item Denote by $\CGH$ the simplical category of compactly generated weak Hausdorff spaces.
    	We will call $\CGH$ the \emph{category of spaces} for simplicity.
    \item For a manifold $M$, denote by $\bOpen{M}{}$ the category of open subsets of $M$ which contain $\partial M$.
      Morphisms of $\bOpen{M}{}$ are inclusions of these open subsets.
      For a manifold $M'$ without boundary, we will simplify the notation as $\Open{M'}{}$.
  \end{enumerate} 
\end{definition}

\begin{definition}\label{def:isoequi}
  A smooth codimension zero embedding $i_v \colon (V, \partial V) \to (W, \partial W)$ between smooth manifolds $V$ and $W$ is an \emph{isotopy equivalence} if there exists a smooth embedding $i_w \colon (W, \partial W) \to (V, \partial V)$ such that $i_v \circ i_w$ and $i_w \circ i_v$ are isotopic to $\id_{(W, \partial W)}$ and $\id_{(V, \partial V)}$ respectively.
\end{definition}

\begin{definition}\label{def:goofun}
  Let $M$ be a smooth manifold of dimension $m$. 
  A \emph{good functor} on $\bOpen{M}{}$ is a functor $F \colon \bOpen{M}{}^{\op} \to \CGH$ of simplicial categories, which satisfies the following conditions:
  \begin{enumerate}
    \item (isotopy invariant) If $i \in \Mor_{\bOpen{M}{}}(V, W)$ is an isotopy equivalence, then $F(i)$ is a weak homotopy equivalence;
    \item For any filtration $\dots V_{i} \subseteq V_{i+1} \dots $ of open subsets of $M$, the  canonical map
      \[
        F(\cup_{i \in \NN}V_i) \to \holim_{i \in \NN} F(V_i)
      \]
      is a weak homotopy equivalence.
  \end{enumerate}
\end{definition}

\begin{definition}\label{not:manifold}
  Let $M$ and $N$ be smooth manifolds.
  \begin{enumerate}
    \item We define $\Emb(M,N)$ to be the space of smooth embeddings of $M$ into $N$.
    \item When $M$ and $N$ are smooth manifolds with boundary, we define $\Emb_{\partial}(M, N)$ to be the space of smooth embeddings $F \colon M \to N$ which preserve the boundary, \ie $F(\partial M) \subseteq \partial N$.
    \item We define  $\Emb_{\partial}(M,N,f)$ to be the space of smooth embeddings that coincide with a given smooth embedding $f \colon M \hookrightarrow N$ near the boundary that is transverse to $\partial N$, \ie
    \[
      \Emb_{\partial}(M,N,f) \coloneqq \{F \in \Emb_{\partial}(M,N) \mid F \pitchfork \partial N, F \text{ and } f \text{ are germ equivalent at } \partial M\}.
    \]
  \end{enumerate}
  We topologise $\Emb(M,N)$ and $\Emb_{\partial}(M, N, f)$ with the compact-open topology.
\end{definition}


\begin{theorem}[Weiss]\label{thm:embgood}
    Let $M$ and $N$ be smooth manifolds with $\dim M \leq \dim N$, and let $f \colon M \hookrightarrow N$ be a fixed smooth embedding with $f \pitchfork \partial N$.
      Then the embedding functor 
      \begin{align*}
        \Emb_{\partial}(\blank, N, f) \colon \bOpen{M}{}^{\op} &\to \CGH \\
         V &\mapsto \Emb_{\partial}(V, N, f)
      \end{align*}
      is a good functor.
\end{theorem}

\begin{proof}
  See \cite[Proposition~1.4]{Wei99}.
\end{proof}

Now we are going to introduce the approximation sequence for good functors produced by manifold calculus.

\begin{definition}\label{def:posetk}
  Denote by $[n]$ the set $\{0, 1, \dots, n\}$.
  \begin{enumerate}
    \item Define the category $\CatPow{[n]}$ as the category whose objects are the subsets of $[n]$ and the morphisms are inclusions of subsets.
    \item Define the full subcategory $\CatPow{[n]}_{\neq\emptyset}$ of $\CatPow{[n]}$, whose objects are the non-empty subsets of $[n]$.
  \end{enumerate}
\end{definition}

\begin{definition}\label{def:polyfun}
  For a manifold $M$ without boundary, a good functor $F$ is a \emph{polynomial functor of degree at most $n$} if for every open subset $U \in \Open{M}{}$ and $A_0, A_1, \dots, A_n$ pairwise disjoint closed subsets of $M$ which lie in $U$, the $(n+1)$-cube
  \begin{align*}
    \chi \colon \CatPow{[n]} &\to \CGH \\
      S &\mapsto F(U \setminus \cup_{i \in S} A_i)
  \end{align*}
  is homotopy cartesian, \ie the natural map $\chi(\emptyset) \to \holim_{S \neq \emptyset}\chi(S)$ is a weak homotopy equivalence. 
  In other words, $F(U) \to \holim_{S \neq \emptyset}F(U \setminus \cup_{i \in S}A_i)$ is a weak homotopy equivalence.
\end{definition}
  
\begin{remark}\label{rmk:bpolyfun}
  One obtains the definition of polynomial functors of degree at most $n$ for manifolds $M$ with boundary by replacing $\Open{M}{}$ with $\bOpen{M}{}$ and requiring that each $A_i$ has empty intersection with $\partial U$ so that $F(U \setminus \cup_{i \in S}A_i)$ is well-defined.
\end{remark}

The name `polynomial functor' may come from the following criterium for polynomial functions.

\begin{lemma}\label{lem:charpoly}
  A smooth function $p \colon \RR \to \RR$ satisfies 
  \begin{equation}\label{eq:poly}
    \sum_{S \subseteq [n]}(-1)^{\#{S}} p \left( \sum_{i \in S}x_i \right) = 0
  \end{equation} 
  for any collection of real numbers $x_0, x_1, \dots, x_n$ if and only if $p$ is a polynomial of degree at most $n$. \qed
\end{lemma}

\begin{definition}\label{def:okcat}
   Let $M$ be a smooth manifold and let $n \in \NN$. 
   Define $\bOpen{M}{n}$ to be the full subcategory of $\bOpen{M}{}$ whose objects are the open subsets $W$ of $M$ that are diffeomorphic to $\mathrm{N}(\partial M) \sqcup (\bigsqcup_{k} \RR^{m})$ with $0 \leq k \leq n$.
   Here $\mathrm{N}(\partial M)$ denotes a (non-fixed) tubular neighbourhood of $\partial M$ in $M$.      
\end{definition}

\begin{definition}\label{def:Tk}
  Let $M$ be a manifold and let $F$ be a good functor. 
  The \emph{$n$-th Taylor approximation} $\pT_n F$ of $F$ is the homotopy right Kan extension
  \[
    \begin{tikzcd}[row sep = huge]
      \bOpen{M}{n}^{\op} \arrow[r, hook, "i_n"] \arrow[d, "F"', ""{name=A, right}] & \bOpen{M}{}^{\op} \arrow[ld, "\pT_n F", ""'{name=B, left}, dashed] \arrow[Leftarrow, from=A, to=B, shorten >=2ex, shorten <=2ex] \\
       \CGH
    \end{tikzcd}
  \]
  of $\restr{F}{\bOpen{M}{n}^{\op}}$ along the inclusion $i_n \colon \bOpen{M}{n}^{\op} \hookrightarrow \bOpen{M}{}^{\op}$, together with the natural transformation $\eta_n \colon F \to \pT_n F$ coming from the universal property of the homotopy right Kan extension.  
  Written as a homotopy limit, the functor $\pT_n F$ is 
  \[
    \pT_n F(V) \simeq \holim_{\mathclap{\substack{W \subseteq V \\ W \in \bOpen{M}{n}}}} F(W).
  \]
\end{definition}

\begin{example}[{\cite[Section~0]{Wei99}}]\label{ex:t1imm}
 Let $M, N$ be smooth manifolds (without boundary).
 The first Taylor approximation $\pT_1 \Emb(V)$ of the embedding functor $\Emb(\blank, N)$ is weakly homotopy equivalent to the immersion functor $\operatorname{Imm}(\blank,N)$, which associates to an open subset $V \subseteq M$ the space of immersions $\operatorname{Imm}(V,N)$.
 In other words, the ``linear'' approximation of $\Emb(M, N)$ is the space $\operatorname{Imm}(M, N)$ of ``local'' embeddings.
\end{example}

\begin{proposition}[Weiss]\label{prop:polyfun}
  Using the notation from Definition~\ref{def:Tk}, the functor $\pT_nF$ toghether with the natural transformation $\eta_n$ has the following properties.
  \begin{enumerate}
    \item The functor $\pT_n F$ is a polynomial functor of degree at most $n$,
    \item For any $V \in \bOpen{M}{n}$, the map $\eta_n(V)$ is a weak homotopy equivalence.
    \item If $F$ is a polynomial functor of degree at most $n$, then $\eta_n$ is a weak equivalence, \ie $\eta_n(V)$ is a weak homotopy equivalence for every $V \in \bOpen{M}{}$.
    \item If $\mu \colon F \to G$ is a natural transformation where $G$ is a polynomial functor of degree at most $n$, then the natural transformation $\mu$ factors through $\pT_n F$, unique up to contractible choices.
  \end{enumerate}
\end{proposition}

\begin{proof}
  See \cite[Theorem~3.9, Theorem~6.1]{Wei99}.
\end{proof}

\begin{remark}\label{rmk:uniqpolyfun}
  In other words, the natural transformation $\eta_n \colon F \to \pT_n F$ is the best approximation of $F$ by a polynomial functors of degree at most $n$ and $\pT_n F$ is unique up to weak homotopy equivalence.
\end{remark}

\begin{definition}\label{def:taylortower}
  Let $M$ be a smooth manifold and let $F \colon \bOpen{M}{} \to \CGH$ be a good functor.
  The \emph{Taylor tower} of $F$ is the tower of natural transformations $r_{i} \colon \pT_i F \to \pT_{i-1} F$ with $i \geq 1$, that are induced by the inclusion $\bOpen{M}{i-1} \hookrightarrow \bOpen{M}{i}$. 
  \[
    \begin{tikzcd}[row sep = huge]
      &  & F \arrow[d, "\eta_n"] \arrow[rd, "\eta_{n-1}"] \arrow[ld, "\eta_{n+1}"]  \arrow[rrrd, "\eta_0"] \arrow[lld] &  &  &   \\
      \dots \arrow[r, "r_{n+2}"'] & \pT_{n+1}F \arrow[r, "r_{n+1}"'] & \pT_n F \arrow[r, "r_{n}"'] & \pT_{n-1} F \arrow[r, "r_{n-1}"'] & \dots \arrow[r, "r_1"'] & \pT_0 F. 
    \end{tikzcd}
  \]
\end{definition}

\begin{definition}
  In the situation of Definition~\ref{def:taylortower},
  The Taylor tower of $F$ \emph{converges} if for every $V \in \bOpen{M}{}$, the canonical map $ \eta(V) \colon F(V) \to \holim_{n \neq 0} \pT_n F(V)$ is a weak homotopy equivalence.
\end{definition}

The Taylor tower of a good functor does not converge in general. 
However, for some embedding functors, we have the following convergence criterium.

\begin{theorem}[Goodwillie--Weiss]\label{thm:convg}
  Let $M$ and $N$ be two smooth manifolds. 
  Then the Taylor tower of the embedding functor $\Emb(\blank,N)$ converges, if $\dim N - \dim M \geq 3$.
\end{theorem}

\begin{proof}
  See \cite[Corollary 2.5]{GW99}.
\end{proof}

Thus, in the codimension $2$ case, where interesting knot theory is developed, this convergence theorem is not applicable. 
Nevertheless, manifold calculus allows us to study the space of knots from a homotopy theoretic viewpoint, as shown in the later sections.
For this purpose, let us first recall the basic notions of knot theory.

\subsection{Long knots and Vassiliev's knot invariants}
Classically, knot theory studies smooth embeddings of $\usphere^1$ into $\usphere^3$ up to isotopy.
A long knot is an embedding from $\I$ to $\RR^2 \times \udisc^1$ coinciding with a fixed linear embedding near the boundary.
We consider long knots instead of knots in this paper because of technical convenience.
For example, the space $\Kca$ of long knots has an $\E_1$-algebra\footnote{We denote by $\E_1$ the little interval operad, which is a concrete model for $A_{\infty}$-operad in $\CGH$.} structure induced by concatenation, \cf \cite[Section 4]{BCKS17} and \cite{Bud08}.
The one point compactification of each long knot induces an isomorphism between $\pi_0(\Kca)$ and $\pi_0(\Emb(\usphere^1, \usphere^3))$.
Thus for the study of knot invariants with values in an abelian group $A$, \ie elements of $\Ho^{0}(\Emb(\usphere^1, \usphere^3); A)$, it does no harm to use long knots instead of knots. 
However, note that $\Emb(\usphere^1, \usphere^3)$ and the space of long knots $\Kca$ have different higher homotopy groups, \cf \cite[Theorem 2.1]{Bud08}.

\begin{definition}\label{def:lknot}
  Fix the embedding $c \colon \I \to \RR^2 \times \udisc^1$, $t \mapsto (0,0, -1 + 2t)$ and define the space $\Kca$ of long knots as $\Emb_{\partial}(\I, \RR^2 \times \udisc^1, c)$.
  Elements of $\Kca$ are called \emph{long knots}.
\end{definition}

For the joy of the reader, see Figure~\ref{fig:butterfly} for an example of a long knot.
\begin{figure}[!ht]
    \includegraphics[width=\textwidth]{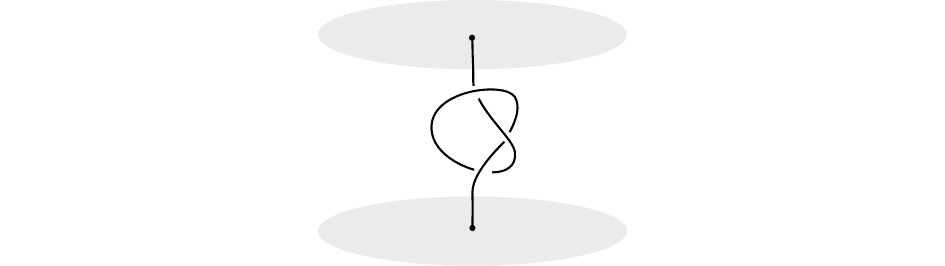}
    \caption{An example of a long knot.}
    \label{fig:butterfly}
\end{figure}

\begin{definition}\label{def:isotopyemb}
  Two long knots $K_0,K_1 \in \Kca$ are called \emph{isotopic}\footnote{In the smooth settings, isotopy and ambient isotopy are equivalent.} if there is a smooth map $F \colon \I \times \I \to \RR^2 \times \udisc^1$ such that $\restr{F}{\I \times \{0\}} = K_0$ and $\restr{F}{\I \times \{1\}} = K_1$, and $\restr{F}{\I \times \{t\}} \in \Kca$ for every $t \in \I$. 
  We call $F$ an isotopy between $K_0$ and $K_1$, and write $K_0 \sim K_1$
\end{definition}

\begin{definition}\label{def:knotinvariant}
  A \emph{knot invariant} with values in a set $R$ is a map $f\colon \pi_0(\Kca) \to R$.
\end{definition}

Classifying knots up to isotopy has always been a central problem in knot theory, so finding computable knot invariants plays an important role.
The knot invariants that we are interested in are the so called finite type invariants, which are a collection of knot invariants discovered by Vassiliev, see \cite{Vas90} for Vassiliev's original approach and \cite{Bar95} for an alternative explanation.
There is a precise definition of Vassiliev invariants using combinatorics, which links the study of Vassiliev invariants to Chern--Simon theory and algebraic structures of Feynman diagrams, \cf \cite{Bar95} and \cite{Kon94}.
For motivation, we sketch Vassiliev's original approach here: 

Instead of focusing on one specific knot invariant, Vassiliev considered the whole set\footnote{Vassiliev used another mapping space instead of $\Kca$, which is homotopy equivalent to $\Kca$. For simplicity, we just write $\Kca$ instead.} $\Ho^0(\Kca; A)$ of all knot invariants with values in a given abelian group $A$.
The main steps of his computation are the following:
\begin{enumerate}
  \item Embed $\Kca$ in the space $\cinfty_{\partial}(\I, \RR^2 \times \udisc^1, c)$ of all smooth maps from $\I$ into $\RR^2 \times \udisc^1$ which are germ equivalent with $c$ on the boundary (Definition~\ref{def:lknot}).
  \item Compute the homology of the complement of $\Kca$ in $\cinfty_{\partial}(\I, \RR^2 \times \udisc^1, c)$. 
  \item Use Alexander duality to obtain $\Ho^{\bullet}(\Kca; A)$, and in particular $\Ho^{0}(\Kca; A)$.
\end{enumerate}

In order to perform step \rom{2} and \rom{3}, Vassiliev finds a filtration by finite dimensional vector spaces $\{\Gamma_i\}_{i \in \NN}$, which approximate the space $\cinfty_{\partial}(\I, \RR^2 \times \udisc^1, c)$.
Intersecting this sequence with $\cinfty_{\partial}(\I, \RR^2 \times \udisc^1, c) \setminus \Kca$ yields a filtration 
\[
\sigma_1 \subseteq \sigma_2 \subseteq \dots \subseteq \sigma_n \subseteq \sigma_{n+1} \subseteq \dots \subseteq \cinfty_{\partial}(\I, \RR^2 \times \udisc^1, c) \setminus \Kca.
\]
Now, Vassiliev computes $\Ho_{\bullet}(\cinfty_{\partial}(\I, \RR^2 \times \udisc^1, c) \setminus \Kca; A)$ via a homological spectral sequence associated to this filtration.
Furthermore, this filtration gives a filtration of the homology groups $\Ho_{\bullet}(\cinfty_{\partial}(\I, \RR^2 \times \udisc^1, c) \setminus \Kca; A)$.
In each of the finite dimensional vector spaces we can apply Alexander duality to obtain a filtration 
\[
  V_1^A \subseteq V_2^A \subseteq \dots \subseteq V_n^A \subseteq V_{n+1}^A \subseteq \dots \subseteq \Ho^{0}(\Kca; A).
\]
Finally, a \emph{Vassiliev invariant of degree at most $n$} with values in $A$ is defined to be an element of~$V_n^A$.

\begin{remark}
  Let $K \in \cinfty_{\partial}(\I, \RR^2 \times \udisc^1, c)$ be a smooth map. 
  We call a point $p \in \im(K)$ a singularity of $K$ if $K^{-1}(p)$ contains more than one element.
  The filtration $(\sigma_i)_{i \geq 1}$ of $\cinfty_{\partial}(\I, \RR^2 \times \udisc^1, c) \setminus \Kca$ arises 
  by distinguishing $K$ by the type and the number of its singularities.
  Thus it appears natural to conjecture that the system of Vassiliev invariants classifies knots.
  On the other hand, it is still open whether Vassiliev invariants detect the unknot.
\end{remark}

\begin{notation}\label{not:embfun}
   We abbreviate the embedding functor $\Emb_{\partial}(\blank, \RR^2 \times \udisc^1, c)$ associated to the space $\Kca$ as $\Emb(\blank)$, since this is the only embedding space we are going to consider in the rest of the text.
\end{notation}

The convergence for the Taylor tower of the embedding functor $\Emb(\blank)$ corresponding to the space $\Kca$ indeed fails, because the set $\pi_0(\Kca)$ is countable, but the homotopy limit of the corresponding Taylor tower can be shown (formally) to be uncountable.
However, the natural transformations $\eta_n$ in the Taylor tower of $\Emb(\blank)$
\begin{equation}\label{eq:mctower}
  \begin{tikzcd}
    &  \Emb(\blank) \arrow[dd, "\eta_n"] \arrow[rdd, "\eta_{n-1}"]  \arrow[rrrdd, "\eta_0"]  \arrow[ldd]&  &  & &\\ \\
      \dots \arrow[r, "r_{n+1}"'] & \pT_n \Emb(\blank) \arrow[r, "r_{n}"'] & \pT_{n-1} \Emb(\blank) \arrow[r, "r_{n-1}"'] & \dots  \arrow[r, "r_1"'] & \pT_0 \Emb(\blank)
  \end{tikzcd}
\end{equation}
do provide us with a sequence of knot invariants, \ie $\pi_{0}\left(\eta_n(\I)\right) \colon \pi_0\left(\Kca\right) \to \pi_0\left(\pT_n \Emb(\I)\right)$ for every $n \in \NN$.
These are actually finite type invariants:

\begin{theorem}[Budney--Connant--Koytcheff--Sinha]\label{thm:BCKS}
  Let $n \in \NN$.
  The map $\pi_0(\eta(\I))$ is an additive finite type knot invariant of degree at most $n-1$.
\end{theorem}

\begin{proof}
  See \cite[Theorem 6.5]{BCKS17}.
\end{proof}

\begin{remark}\label{rmk:monoid-BCKS}
  Let $A$ be an abelian group.
  An \emph{additive} knot invariant is a abelian monoid homomorphism $\pi_0(\Kca) \to A$, where the abelian\footnote{The connected sum of knots is up to isotopy abelian because one can shrink one knot in the connected sum and slide it through the other knot in the connected sum.} monoid structure of $\pi_0(\Kca)$ is induced by connected sum of knots.
  Thus, one non-trivial point in Theorem~\ref{thm:BCKS} is to give $\pi_0\left(\pT_n \Emb(\I)\right)$ an abelian monoid structure such that it is compatible with the connected sum of knots. 
  The authors of \cite{BCKS17} solve this by defining compatible $\E_1$-algebra structures on the spaces $\Kca$ and $\pT_n \Emb(\I)$, \cf \cite[Section 4]{BCKS17}.
  For more survey on the operadic structures on the space of long knots (not necessarily in codimension $2$), \cf \cite{Bud08}, \cite{Sin06} and \cite{DH12}.
\end{remark}

\begin{conjecture}[Budney--Connant--Koycheff--Sinha]\label{conj:BCKS}
  Let $n \in \NN$.
  The map $\pi_0(\eta_{n}(\I))$ is the universal additive finite type knot invariant of degree at most $n-1$.
\end{conjecture}

We refer the readers to \cite{Hab00}, \cite{CT04}, \cite{CT04b} and \cite{Sta00} for various geometric and combinatorial descriptions of universal additive Vassiliev invariants.
If the conjecture is true, then we expect $\pi_0\left(\pT_n \Emb(\I)\right)$ to be isomorphic to the abelian groups generated by certain combinatorial diagrams.
The rest of the text aims at explaining one method of understanding the homotopy groups (at least $\pi_0$) of $\pT_n \Emb(\I)$, namely, the computation of a Bousfield--Kan homotopy spectral sequence of the following tower of fibration:
\[
\cdots \xrightarrow{r_{n+1}(\I)} \pT_n \Emb(\I) \xrightarrow{r_{n}(\I)} \pT_{n-1} \Emb(\I) \xrightarrow{r_{n-1}(\I)} \cdots \xrightarrow{r_{1}(\I)} \pT_0 \Emb(\I).
\]
In order to do this, we introduce in the next section a cosimplicial space associated to the embedding functor $\Emb(\blank)$.


\section{A cosimplicial model constructed by Goodwillie}\label{sec:cosm}
\counterwithin{equation}{section}

Goodwillie observed that a good functor on $\bOpen{\I}{}$ corresponds to an augmented cosimplicial space, and vice versa, \cf \cite[Section~5]{GKW01}.
The cosimplicial spaces arise naturally in this way enjoy nice properties and can be considered as cosimplicial models for the Taylor tower of the corresponding good functors (Definition \ref{def:cosmodel}). 
Because certain details are left out in the paper \cite{GKW01} for the construction of the cosimplicial spaces and their properties, we reformulate this correspondence in terms of equivalence of simplicial functor categories and give proofs in full detail, using some results from \cite{AF15}.
This equivalence further facilitates the computations in Section~\ref{sec:hssknotcal}.

\begin{definition}\label{def:catsind}
  We define the following two categories.
  \begin{enumerate}
    \item The \emph{simplex category} $\CatsInd$ consists of the objects $[n] = \{0,1,\dots, n\} \subseteq \NN$ with $n \geq 0$, and the morphisms are the order-preserving maps.
    \item Denote by $\CatsInd_{+}$ the category of finite, totally ordered sets.
      The morphisms are order-preserving maps.
  \end{enumerate}  
\end{definition}

\begin{remark}
  The simplex category $\CatsInd$ is equivalent to the category of non-empty finite totally ordered sets, which we also denote, by abuse of notation, by $\CatsInd$.
\end{remark}

\begin{definition}\label{def:cosspace}
  \begin{enumerate}
    \item A \emph{cosimplicial space} is a functor $X^{\bullet} \colon \CatsInd \to \CGH$.
    \item An \emph{augmented cosimplicial space} $Y^{\bullet}_{+}$ is a functor $Y^{\bullet}_{+} \colon \CatsInd_{+} \to \CGH$.
  \end{enumerate}
\end{definition}

\begin{notation}
  Let $X^{\bullet}$ be a cosimplicial space.
  Denote by $X^{k}$ the value $X^{\bullet}([k])$, for $k \geq 0$.
\end{notation}

\begin{notation}
  By restricting an augmented cosimplicial space $Y^{\bullet}_{+}$ to the subcategory $\CatsInd$, we obtain the \emph{associated cosimplicial space}, which we denote as $Y^{\bullet}$.
\end{notation}

\begin{convention}
  By the \emph{totalisation} $\Tot X^{\bullet}$ of a cosimplicial space $X^{\bullet}$ we always mean the totalisation $\Tot \widetilde{X^{\bullet}}$ of a fibrant replacement $\widetilde{X^{\bullet}}$ of $X^{\bullet}$, with respect to the model structure introduced in \cite[Section X.4.6]{BK72}.
  Similarly, by the \emph{$n$-th partial totalisation} $\Tot^{n} X^{\bullet}$ of $X^{\bullet}$ we always mean $\Tot^{n} \widetilde{X^{\bullet}}$.
\end{convention}

\begin{remark}[{\cite[Chapter XI.4.4]{BK72}}]
  We have natural weak homotopy equivalences
  \begin{align*}
    \Tot X^{\bullet} &\simeq \holim \widetilde{X^{\bullet}} \\
    \Tot^{n} X^{\bullet} &\simeq \holim_{k \leq n} \widetilde{X^{k}}
  \end{align*}
\end{remark}

\begin{definition}\label{def:finOpen}
  Let $\fbOpen{\I}$ be the full subcategory of $\bOpen{\I}{}$ whose objects are the open subsets of $\I$ that contain $\partial \I$ and have only finitely many path connected components, \ie
  \[
    \Ob\left(\fbOpen{\I}\right) \coloneqq \{\I\} \cup \bigcup_{n \geq 0} \Ob\left(\bOpen{\I}{n}\right)
  \]
\end{definition}

\begin{proposition}\label{prop:finOpen}
  A good functor on $\bOpen{\I}{}^{\op}$ is determined up to weak homotopy equivalence by its restriction on $\fbOpen{\I}^{\op}$.
\end{proposition}

\begin{proof}
  This follows by Definition~\ref{def:goofun}.\rom{2} of good functor.
\end{proof}

\begin{remark}\label{rmk:finOpeniso}
  The restriction of a good functor on $\bOpen{\I}{}^{\op}$ to $\fbOpen{\I}^{\op}$ is isotopy invariant in the sense of Definition~\ref{def:goofun}.i).
  An isotopy invariant functor on $\fbOpen{\I}^{\op}$ fullfills Definition~\ref{def:goofun}.ii) automatically.
\end{remark}

In the following, we reformulate and prove the correspondence described in \cite[Construction 5.1.1]{GKW01} between augmented cosimplicial spaces and good functors from $\fbOpen{\I}^{\op}$ to the category of spaces.
We need first few notation from the theory of simplicially enriched categories~\cite[A.3]{Lur09}.

\begin{situation}
  Denote by $\Category{Cat}_{\ssimplex}$ the category of simplicially enriched categories. 
  In other words, objects of $\Category{Cat}_{\ssimplex}$ are simplicially enriched categories and morphisms are simplicial functors. 
  A functor $F \colon \Cca \to \Dca$ in $\Category{Cat}_{\ssimplex}$ is a \emph{weak equivalence }if 
  \begin{enumerate}
    \item for every object $Y \in \Dca$, there exists $X \in \Cca$ such that $F(X) = Y$ in the homotopy category $\ho\Dca$ of $\Dca$, and 
    \item $F$ induces an equivalence $\Map_{\Cca}(X, X') \xrightarrow{\simeq} \Map_{\Dca}(F(X), F(X'))$ of mapping spaces for every pair $(X, X')$ of object in $\Cca$.
  \end{enumerate}
  In fact, the category $\Category{Cat}_{\ssimplex}$ admits a left proper combinatorial model structure with this notion of weak equivalences \cite[Proposition~A.3.2.4]{Lur09}.
  Therefore, we can form the homotopy category $\ho\Category{Cat}_{\ssimplex}$ of $\Category{Cat}_{\ssimplex}$. 
\end{situation}

\begin{theorem}\label{thm:goodwilliecos}
  Let $\kappa \colon \fbOpen{\I}^{\op} \to \CatsInd_{+}$ be the functor $\kappa(V) \coloneqq \pi_0(\I \setminus V)$.
  Then $\kappa$ induces a bijection 
  \[
  \Hom_{\ho\Category{Cat}_{\ssimplex}}\left(\CatsInd_{+}, \CGH\right) \xrightarrow{\kappa \circ \blank} \Hom_{\ho\Category{Cat}_{\ssimplex}}^{\mathrm{g}}\left(\fbOpen{\I}^{\op}, \CGH\right)
  \]
  where $\Hom_{\ho\Category{Cat}_{\ssimplex}}^{\mathrm{g}}\left(\fbOpen{\I}^{\op}, \CGH\right)$ is the subset of good functors. 
\end{theorem}

\begin{remark}
  A more elegant way to formulate this is to use the language of $\infty$-categories: the functor $\kappa$ induces an equivalence 
  \[
  \infFun\left(\Nv\left(\CatsInd_{+}\right), \infHType \right) \xrightarrow{\kappa \circ \blank} \infFun^{g}\left(\Nv\left(\fbOpen{\I}^{\op}\right), \infHType\right)
  \]
  of $\infty$-category of functors. 
  Here, $ \infFun^{g}$ denotes the $\infty$-category of good functors, $\Nv$ denotes the nerve functor, and $\infHType$ denotes the $\infty$-category of homotopy types which can be obtained from $\CGH$ by inverting all weak homotopy equivalences~\cite[Remark~1.2.16.3]{Lur09} . 
  In particular, here we have an ``enriched''-version of the statement of the above theorem, instead of only talking in the homotopy category $\ho\Category{Cat}_{\ssimplex}$.
  We will see later in the proof that Theorem~\ref{thm:goodwilliecos} is a consequence of the universal property of simplicial localisation, which can be in general only formulated in $\ho\Category{Cat}_{\ssimplex}$ because the construction itself involves various cofibrant and fibrant replacements.
  However, we choose not to go further into the $\infty$-categorical setting, since it is not necessary for our later applications.
\end{remark}

\begin{corollary}[{\cite[Remark 5.1.3]{GKW01}}]\label{cor:goodwilliecosprop}
	Let $F \colon \fbOpen{\I}^{\op} \to \CGH $ be a good functor.
	Denote by $\Ffr^{\bullet}_{+}$ an augmented cosimplicial space such that $\kappa \circ \Ffr^{\bullet}_{+} \simeq F$.
	Then the associated cosimplicial space $\Ffr^{\bullet}$ has the following properties:
	\begin{enumerate}
	\item For $V \in \fbOpen{\I}^{\op}$, we have that $\Ffr^{\bullet}\left({\pi_0(\I \setminus V)}\right) \simeq F(V)$.
	\item We have $\Tot^{n}\Ffr^{\bullet} \simeq \pT_n F(\I)$ and $\Tot \Ffr^{\bullet} \simeq \holim_{n \geq 0} \pT_n F(\I)$.
	\end{enumerate}
\end{corollary}

\begin{definition}\label{def:cosmodel}
	Using the notations of the above corollary, we call a cosimplicial space~$\Ffr^{\bullet}$ satisfying the i) and ii) of Corollary \ref{cor:goodwilliecosprop} a \emph{cosimplicial space associated to the good functor~$F$}.
	
	Because of $\Ffr^{\bullet}$ enjoys the property of Corollary \ref{cor:goodwilliecosprop}.ii), we call $\Ffr^{\bullet}$ a \emph{cosimplicial model for the tower of fibrations}
		\[
		\cdots \xrightarrow{r_{n+1}(\I)} \pT_n F(\I) \xrightarrow{r_{n}(\I)} \pT_{n-1} F(\I) \xrightarrow{r_{n-1}(\I)} \cdots \xrightarrow{r_{1}(\I)} \pT_0 F(\I),
		\]
		obtained by evaluation of the Taylor tower of $F$ on $\I$.
\end{definition}

\begin{remark}\label{rmk:emb&conf}
  Let $n = \#(\pi_0(\I \setminus V ))-1$.
  Denote by $\Embcos^{\bullet}$ a cosimplicial space associated to $\Emb(\blank)$.
  Corollary~\ref{cor:goodwilliecosprop}.\rom{1} implies that $\Embcos^{n}$ is weakly homotopy equivalent to the Cartesian products of the configuration space of $n$-points in $\RR^2 \times \udisc^1$ (points of embeddings) and $n$-copies of $\usphere^2$ (tangent vectors at the embedded points).
  We will use this relation to configuration spaces in Section \ref{sec:hssembcos}. 
\end{remark}

Now we work towards proofs of Theorem \ref{thm:goodwilliecos} and Corollary \ref{cor:goodwilliecosprop}, which is, to the best of our knowledge, not available in the literature. 
Let us begin by introducing several categories.

\begin{definition}
  \begin{enumerate}
    \item Define the category $\CatMan{m}$ of smooth oriented $m$-dimensional manifolds. 
      Objects of $\CatMan{m}$ are smooth oriented manifolds of dimension $m$, and the morphisms are the orientation-preserving smooth embeddings.
    \item Define the simplicial category $\infCatMan{m}$ of smooth oriented $m$-dimensional manifolds. 
      This category has the same objects as $\CatMan{m}$, and the morphisms are spaces of orientation-preserving smooth embeddings, equipped with the compact-open topology.
    \item Define the full subcategory $\CatDisc{m}$ of $\CatMan{m}$ whose objects are finite disjoint unions of $\RR^{m}$ and $\RR_{\geq 0} \times \RR^{m-1}$.
    \item Define the full simplicial subcategory $\infCatDisc{m}$ of $\infCatMan{m}$ that has the same objects as $\CatDisc{m}$.
    \item Let $M$ be a smooth oriented manifold of dimension $m$. 
      Define the category 
      \[
      \CatDiscOver{m}{M} \coloneqq \CatDisc{m} \underset{\CatMan{m}}{\times} \CatManOver{m}{M},
      \]
      where $\CatManOver{m}{M}$ is the slice-category over $M$.
      Objects of $\CatDiscOver{m}{M}$ are embeddings of finite disjoint unions of $\RR^m$ and $\RR_{\geq 0} \times \RR^{m-1}$ into $M$.
    \item Let $M$ be a smooth oriented manifold of dimension $m$. 
      Define the simplicial category 
      \[
      \infCatDiscOver{m}{M} \coloneqq \infCatDisc{m} \underset{\infCatMan{m}}{\times} \infCatManOver{m}{M},
      \]
      where $\infCatManOver{m}{M}$ is the over category over $M$.
    \item Let $M$ be a smooth oriented manifold of dimension $m$. 
      Define the subcategory $\Category{Isot}_{\mathrm{m} / \mathrm{M}}$ of $\CatDiscOver{m}{M}$ which has the same objects as $\CatDiscOver{m}{M}$, but only the morphisms that are isotopy equivalences. 
  \end{enumerate}
\end{definition}

\begin{situation}
  Let $\Cca$ and $\Dca$ be simplicially enriched categories and $W$ be a set of morphisms in $\Cca$.
  Denote by $\Cca[W^{-1}]$ the simplicial localisation of $\Cca$ with respect to $W$. 
  It is a simplicial category together with a functor $L \colon \Cca \to \Cca[W^{-1}]$.
  There are several models \cite{DK80}\footnote{In the model for simplicial localisation developed in \cite{DK80}, the localisation functor $L$ exists in the homotopy category $\ho\Category{Cat}_{\ssimplex}$}, \cite{DK802} and \cite[A.3.5]{Lur09} etc.~for simplicial localisations, which are all equivalent after suitable fibrant or cofibrant replacements \cite{Ste17}. 
  Simplicial localisations enjoys the following universal property~\cite[Proposition~A.3.5.5]{Lur09}: the functor $L$ induces an injective map 
  \begin{equation}\label{eq:univ-prop-simplicial-loc}
     \Hom_{\ho\Category{Cat}_{\ssimplex}}\left(\Cca[W^{-1}], \Dca\right) \to \Hom_{\ho\Category{Cat}_{\ssimplex}}\left(\Cca, \Dca\right)
  \end{equation}
  whose image is the subset of functors $\Cca \to \Dca$ that sends $W$ to equivalences in $\Dca$.
\end{situation}

\begin{proposition}[Ayala--Francis]\label{prop:ayala}
  The canonical functor 
  \[
  \CatDiscOver{m}{M} \hookrightarrow \infCatDiscOver{m}{M}
  \]
  induces an equivalence of simplicial categories
  \[
    \CatDiscOver{m}{M} [\Category{Isot}_{\mathrm{m} / \mathrm{M}}^{-1}] \simeq \infCatDiscOver{m}{M}.
  \]
\end{proposition}

\begin{proof}
  See \cite[Proposition 2.19]{AF15}.
\end{proof}

\begin{remark}
  In \cite{AF15}, the authors uses the language of $\infty$-categories (quasi-categories). 
  For a translation between simplicial categories and $\infty$-categories, see~\cite[Section 1.1.3, 1.1.4 and 1.1.5]{Lur09}.
  For a comparison between simplicial and $\infty$-categorical localisations, see~\cite{Ste17}.
\end{remark}

For our application, we consider $m = 1 $ and $M = \I$.

\begin{definition}\label{def:surjdisk}
  Define the subcategory $\CatDiscOver{1}{\I}^{\partial}$ of $\CatDiscOver{1}{\I}$ whose objects are the embeddings such that the boundary $\partial \I$ of $\I$ is in the image. 
  
  In the same way, let $\Category{Isot}_{\mathrm{1} / \mathrm{\I}}^{\partial}$ be the subcategory of $\Category{Isot}_{\mathrm{1} / \mathrm{\I}}$ whose objects are the embeddings such that $\partial \I$ is in their images.
  Explicitly, objects of $\CatDiscOver{1}{\I}^{\partial}$ are embeddings of the form $$[0, \epsilon) \sqcup \bigsqcup_{k = 1}^n \RR \sqcup (1-\epsilon', 1] \hookrightarrow \I, \text{ for }n \in \NN.$$
\end{definition}

Using the same proof strategy as Proposition~\ref{prop:ayala}, we have

\begin{corollary}\label{cor:ayalacor}
  The canonical functor $\CatDiscOver{1}{\I}^{\partial} \to \infCatDiscOver{1}{\I}^{\partial}$ induces an equivalence of simplicial categories
  \begin{equation*}\tag*{\qedsymbol}
    \CatDiscOver{1}{\I}^{\partial}\left[\left( \Category{Isot}_{\mathrm{1} / \mathrm{\I}}^{\partial} \right)^{-1} \right] \simeq \infCatDiscOver{1}{\I}^{\partial}.
  \end{equation*}
\end{corollary}

\begin{proposition}\label{prop:catdisksequi}
  \begin{enumerate}
    \item (Ayala--Francis) The functor
      \begin{align*}
        S \colon \left( \infCatDiscOver{1}{\I}^{\partial} \right)^{\op} &\to \CatsInd_{+} \\
        \left(i \colon V \hookrightarrow \I \right) &\mapsto \pi_0(\I \setminus i(V))
      \end{align*}
      is an equivalence of simplicial categories.
    \item The functor $\mathrm{Im} \colon \CatDiscOver{1}{\I}^{\partial} \to \fbOpen{\I}{}$, $(i \colon V \hookrightarrow \I) \mapsto \im(i)$ is an equivalence of ordinary categories.
    	Let $\Category{Isot}(\I)$ be the subcategory of $\fbOpen{\I}{}$ which has the same objects as $\fbOpen{\I}{}$ with only the morphisms that are isotopy equivalences.
      Then $\restr{\mathrm{Im}}{\Category{Isot}_{\mathrm{1} / \mathrm{\I}}^{\partial}} \colon \Category{Iso}_{\mathrm{1} / \mathrm{\I}}^{\partial} \to \Category{Isot}(\I)$ is an equivalence of categories.
  \end{enumerate}
\end{proposition}

\begin{proof}
 \rom{1} We can prove that the functor is essentially surjective and fully faithful.
 This is true since the every connected component of each space of morphism of $\infCatDiscOver{1}{\I}^{\partial}$ is contractible.
 See also \cite[Lemma~3.11]{AF15}. 
 
 \rom{2} First $\mathrm{Im}$ is essentially surjective, because the boundary of $\I$ is in the image of $i$ for any $i \in \Category{Isot}_{\mathrm{1} / \mathrm{\I}}^{\partial}$. 
    For any two objects $i_1 \colon V_1 \hookrightarrow \I$ and $i_2 \colon V_2 \hookrightarrow \I$ in $\CatDiscOver{1}{\I}^{\partial}$, the morphism set $\Mor_{\CatDiscOver{1}{\I}^{\partial}}(i_1, i_2)$ is either empty or a one element set.
    Also the morphism set~$\Mor_{\bOpen{\I}{}}(\im(i_1), \im(i_2))$ is either empty or has one element.
    Thus $\mathrm{Im}$ is fully faithful.
    Similarly it follows that $\restr{\mathrm{Im}}{\Category{Isot}_{\mathrm{1} / \mathrm{\I}}^{\partial}}$ is an equivalence.
\end{proof}

\begin{corollary}\label{cor:chaincatequi}
  We have the following equivalences of simplicial categories
  \begin{equation*}\tag*{\qedsymbol}
    \fbOpen{\I}{}[\Category{Iso}(\I)^{-1}] \simeq \CatDiscOver{1}{\I}^{\partial}\left[\left( \Category{Isot}_{\mathrm{1} / \mathrm{\I}}^{\partial} \right)^{-1} \right] \simeq \infCatDiscOver{1}{\I}^{\partial} \overset{S}{\simeq} \CatsInd_{+}^{\op}.
  \end{equation*}
\end{corollary}

\begin{proof}[Proof of Theorem \ref{thm:goodwilliecos}]
	Recall the functor $\kappa \colon \fbOpen{\I}^{\op} \to \CatsInd_{+}$ with $\kappa(V) = \pi_0(\I \setminus V)$ from Theorem \ref{thm:goodwilliecos}.
	This functor fits into the equivalences of Corollary \ref{cor:chaincatequi}, \ie we have the following commutative diagram
	\begin{equation}\label{diag:equivfunctors}
		\begin{tikzcd}
	                   & \CatsInd_{+}^{\op}           & \infCatDiscOver{1}{\I}^{\partial} \arrow[l, "S"', "\simeq"]           \\
	\fbOpen{\I} \arrow[r, "L"] \arrow[ru, "\kappa"] & \fbOpen{\I}[\Category{Iso}(\I)^{-1}] \arrow[u, "\kappa_{\mathrm{loc}}"'] & \CatDiscOver{1}{\I}^{\partial}\left[\left( \Category{Isot}_{\mathrm{1} / \mathrm{\I}}^{\partial} \right)^{-1} \right] \arrow[l, "\simeq"] \arrow[u, "\simeq"']
	\end{tikzcd}
	\end{equation}
	where $\kappa_{\mathrm{loc}}$ is the induced map of $\kappa$ by the universal property of the localisation functor $L$, since $\kappa$ sends an isotopy equivalence morphisms in $\fbOpen{\I}$ to isomorphisms in $\CatsInd_{+}$.
	
	By the universal property~(\ref{eq:univ-prop-simplicial-loc}) of localisation and isotopy invariance of good functors, precomposing with the localisation functor $L$ induces an bijection
	\[
	\Hom_{\ho\Category{Cat}_{\ssimplex}}^{\mathrm{g}}\left(\fbOpen{\I}[\Category{Iso}(\I)^{-1}]^{\op}, \CGH\right) \xrightarrow{L \circ \blank} \Hom_{\ho\Category{Cat}_{\ssimplex}}\left(\fbOpen{\I}^{\op}, \CGH \right)
	\]
	of the set of homotopy classes of functors.
	By the equivalences of Corollary \ref{cor:chaincatequi}, the map $\kappa_{\mathrm{loc}}$ in the diagram (\ref{diag:equivfunctors}) also induces an equivalence
	\[
	\Fun\left(\CatsInd_{+}, \CGH\right) \xrightarrow{\kappa_{\mathrm{loc}} \circ \blank} \Fun\left(\fbOpen{\I}[\Category{Iso}(\I)^{-1}]^{\op}, \CGH\right) 
	\]
	of simplicial functor categories.
	
	Therefore, the map $\kappa$ induces an bijection
	\[
	\Hom_{\ho\Category{Cat}_{\ssimplex}}\left(\CatsInd_{+}, \CGH\right) \xrightarrow{\kappa \circ \blank} \Hom_{\ho\Category{Cat}_{\ssimplex}}^{\mathrm{g}}\left(\fbOpen{\I}^{\op}, \CGH\right).
	\]
\end{proof}

\begin{proof}[Proof of Corollary \ref{cor:goodwilliecosprop}]
	Let $\Ffr^{\bullet}_{+}$ be an augmented cosimplicial space such that $\kappa \circ \Ffr^{\bullet}_{+} \simeq F$. 
	Then, part \rom{1} follows by composition of functors. 
	
	As for the proof of Corollary \ref{cor:goodwilliecosprop}.ii), We have 
	\[
	\Tot^{n} \Ffr_{\bullet} = \holim_{\ssimplex^{\bullet \leq n}} \Ffr_{\bullet} \simeq \holim_{L(\bOpen{\I}{n})} F_{\mathrm{loc}} \simeq \holim_{\bOpen{\I}{n}} F
	\]
	where $L$ is the localisation functor in diagram (\ref{diag:equivfunctors}) and $F_{\mathrm{loc}}$ denotes the induced functor from $\fbOpen{\I}[\Category{Iso}(\I)^{-1}]^{\op}$ to $\CGH$ by $F$.
	The first equality is by definition.
	The second weak homotopy equivalence comes from the equivalences of categories from Corollary~\ref{cor:chaincatequi}.
	The third weak homotopy equivalence follows by observing that under the functor
	\[
	\Fun\left(\fbOpen{\I}[\Category{Iso}(\I)^{-1}]^{\op}, \CGH\right) \xrightarrow{L \circ \blank} \Fun\left(\fbOpen{\I}^{\op}, \CGH\right)
	\]
	a homotopy limit cone over a functor $\fbOpen{\I}[\Category{Iso}(\I)^{-1}]^{\op} \to \CGH$ restricts to a homotopy limit cone over a functor $\fbOpen{\I}^{\op} \to \CGH$.
\end{proof}

\begin{remark}\label{rmk:othermodel}
  There is another cosimplicial model $C^{\bullet}$ for the above tower of fibrations, constructed via a compactification of $\Conf_{n}(\RR^2 \times \udisc^1)$, \cf \cite[Section 3]{BCKS17} and \cite[Section 6]{Sin09}. 
  Compared with $\Embcos^{\bullet}$, the cosimplicial space $C^{\bullet}$ has the advantage that it is geometric and various versions of $C^{\bullet}$ have already been used in context concerning finite type invariants of $3$-manifold, \cf \cite{AS94} and \cite{BT94}.
  Using this cosimplicial space $C^{\bullet}$, the authors of \cite{BCKS17} give the space $\Tot^nC^{\bullet} \simeq \pT_n \Emb(\I)$ an $\E_1$-algebra structure, which is used to define the abelian group structure on $\pi_0(\pT_n \Emb(\I))$ mentioned in Remark~\ref{rmk:monoid-BCKS}, \cf \cite[Corollary~4.13, Section~5.7]{BCKS17}.
  We are exploring whether we can define similar multiplicative structure on the cosimplicial space $\Embcos^{\bullet}$.
\end{remark}

\begin{remark}
  One way to generalise Theorem~\ref{thm:goodwilliecos} in higher dimension, \ie good functors from $\bOpen{M}{}^{\op}$ with $\dim M > 1$, is to consider configuration categories $\operatorname{con}(M)$ associated to a smooth manifold $M$.
  See \cite{BW18} for a detailed explanation on this subject.
\end{remark}

\section{An integral homotopy spectral sequence for $\Embcos^{\bullet}$}\label{sec:hssembcos}

\counterwithin{equation}{subsection}

To a cosimplicial space $\Embcos^{\bullet}$ (Definition \ref{def:cosmodel}) associated to the functor $\Emb(\blank)$, we can associate the Bousfield--Kan homotopy spectral sequence $\{E_{p,q}\}_{q \geq p \geq 0}$ with integral coefficients, \cf \cite[Chapter~X]{BK72}. 
In this section, we will first briefly recall some properties of this spectral sequence (Section~\ref{sec:cosss}) and then give a concrete computation of the $d^1$-differentials that map into the diagonal terms (Section~\ref{sec:hssknotcal}). 
For the computation, we make use of the calculation of the homotopy groups for $\Conf_n\left(\RR^2 \times \udisc^1 \right)$, which we will recall in Section~\ref{sec:htpgrconf}.


\subsection{A spectral sequence for cosimplicial spaces}\label{sec:cosss}

We begin by introducing techniques that we need for the computation of Bousfield--Kan spectral sequences.
The main reference for this section is \cite[Chapter X]{BK72}.

\begin{notation}\label{not:cosimplicial}
  Let $X^{\bullet}\colon \CatsInd \to \CGH $ be a cosimplicial space. 
  For $0 \leq i \leq n$, denote its coface map by $\delta^i \colon X^{n-1} \to X^{n}$ and its codegeneracy maps by $s^i \colon X^{n+1} \to X^{n}$.
\end{notation}

Given a cosimplicial space $X^{\bullet} \colon \CatsInd \to \CGH$, there is a tower of fibrations (\cf \cite[Chapter~6, Section~6.1]{BK72})
\begin{align}\label{ali:tottower}
  \cdots \to \Tot^{n+1}X^{\bullet} \to \Tot^{n}X^{\bullet} \to \cdots \to \Tot^{1}X^{\bullet} \to \Tot^{0}X^{\bullet}.  
\end{align}
Denote by $L^{n+1}X^{\bullet}$ the homotopy fibre of $\Tot^{n+1}X^{\bullet} \to \Tot^{n}X^{\bullet}$ and $L^0X^{\bullet} \coloneqq \Tot^0X^{\bullet}$.
Applying Bousfield--Kan homotopy spectral sequence, \cf \cite[Section~X.6]{BK72}, to the tower of fibrations (\ref{ali:tottower}), we obtain a spectral sequence approximating the homotopy groups of $\Tot X^{\bullet}$, whose first page is given by
\[
  E^1_{p,q} = \pi_{q-p}(L^{p}X^{\bullet}),
\]
where $q \geq p \geq 0$, and $E_{p, q}^{1} = 0$ otherwise.
With the help of the cosimplicial structure, we can calculate the homotopy groups of the spaces $L^{p}X^{\bullet}$ and the differentials $d^1 \colon E_{p-1, p}^1 \to E_{p, p}^1$ in terms of the homotopy groups of $X^{p}$.

\begin{proposition}[Bousfield--Kan]
Given a cosimplicial space $X^{\bullet} \colon \CatsInd \to \CGH$, we have 
\[
  \pi_{q-p}(L^p X^{\bullet}) \cong \pi_q \left(X^{p} \cap \bigcap_{i = 0}^{p-1} \ker(s^i)\right) \cong \pi_{q}(X^{p}) \cap \bigcap_{i=0}^{p-1}\ker(s^{i}_{\ast}),
\]
where the push-forward $s^i_{\ast} \colon \pi_q(X^{p}) \to \pi_q(X^{p-1})$ is induced by the codegeneracy maps $s^i$.
\end{proposition}

\begin{proof}
  See \cite[Section~X.6.2]{BK72}.
\end{proof}

\begin{proposition}[Bousfield--Kan]\label{prop:cosimplicialss}
Given a cosimplicial space $X^{\bullet}$, the first page of the Bousfield--Kan homotopy spectral sequence of $X^{\bullet}$ is given by
\[
  E^1_{p,q} \cong \pi_{q}(X^{p}) \cap \bigcap_{i=0}^{p-1}\kernel(s^{i}_{\ast}),
\]
where $q \geq p \geq 0$, and the push-forward $s^i_{\ast} \colon \pi_q(X^{p}) \to \pi_q(X^{p-1})$ is induced by the codegeneracy maps $s^i$.
The differential $ d^1 \colon E_{p,q}^1 \to E_{p+1,q}^1$ on the first page is given by
\[
  x \mapsto \sum_{i=0}^{p+1}(-1)^i \delta^i_{\ast}(x),
\]
where the push-forward $\delta^i_{\ast} \colon \pi_q(X^{p}) \to \pi_q(X^{p+1})$ is induced by the coface maps $\delta^i$ on~$X^{p}$. 
\end{proposition}

\begin{proof}
  See \cite[Chapter~X.7]{BK72}.
\end{proof}


\subsection{Homotopy groups of $\Conf_{n}(\RR^2 \times \udisc^1)$}\label{sec:htpgrconf}
From Section~\ref{sec:cosss} we see that we need to compute the homotopy groups of $\Embcos^{n}$ for $n \geq 0$, in order to compute the Bousfield-Kan spectral sequence associated to the cosimplicial model $\Embcos^{\bullet}$.
By Remark~\ref{rmk:emb&conf} we know that $\Embcos^{n}$ relates closely to the configuration spaces of $n$ points in $\RR^2 \times \udisc^1$.
Therefore, let us gather some information about the homotopy groups of configurations spaces in this section.
The main reference for this section is \cite{FN62} and \cite{FH01}.

\begin{definition}\label{def:config}
 Let $M$ be a smooth manifold (possibly with boundary). 
 Define the \emph{configuration space} $\Conf_n(M)$ of $n \geq 1$ points on $M$ as 
  \[
    \Conf_n(M) \coloneqq \{(x_1, \dots, x_n) \in (M \setminus \partial M)^n \mid x_i \neq x_j \textnormal{ for } i \neq j\}.
  \]
\end{definition}

Now we focus on $\Conf_n (\RR^2 \times \udisc^1)$ for $n \geq 0$.

\begin{convention}
  We define\footnote{We make this conventions because we will see in the next section that $\Embcos_{n} \simeq \Conf_n (\RR^2 \times \udisc^1) \times \left( \usphere^2 \right)^n$ for $n \geq 0$.} $\Conf_0(\RR^2 \times \udisc^1) \coloneqq \{(0,0,-1), (0,0,1)\} \subseteq \partial(\RR^2 \times \udisc^1)$. 
\end{convention}

\begin{situation}\label{not:fixpoints}
  Let us fix the following points of $\RR^2 \times \udisc^1$.
  Define $e \coloneqq (1,0,0) \in \RR^2 \times \udisc^1$, and $q_1 \coloneqq (0,0,0)$, and $q_k= q_1 + 4(k-1)e$ for $k \geq 1$
  Also define the set of points $Q_0 \coloneqq \emptyset$ and $Q_k \coloneqq \{q_1, q_2, \dots, q_k\}$.
\end{situation}

\begin{theorem}[Fadell--Neuwirth]\label{thm:crosssecconfig}
  For $n \geq 2$ and $n \geq k \geq 0$, the map 
  \begin{align*}
    \pr_{k,n} \colon \Conf_n(\RR^2 \times \udisc^1 \setminus Q_k) &\to \RR^2 \times \udisc^1 \setminus Q_k \\
    (x_1, x_2, \cdots, x_n) &\mapsto x_1
  \end{align*}
  is a fibre bundle whose fibre is homeomorphic to $\Conf_{n-1}(\RR^2 \times \udisc^1 \setminus Q_{k+1})$.
 For $k \geq 0$, the map $\pr_{k,n}$ admits a cross section\footnote{The case $k=0$ works since we are looking at Euclidean spaces.}.
\end{theorem}

\begin{proof}
  See \cite[Theorem~2]{FN62}.
\end{proof}

Thus we can compute $\pi_{\ast}(\Conf_n(\RR^2 \times \udisc^1))$ inductively via the splitting long exact sequences for the fibre bundles $\pr_{k, n}$ for $0 \leq k \leq n$ and $n \geq 2$.
And we can conclude the following corollary.

\begin{corollary}[Fadell--Neuwirth]\label{cor:piconfig}
  For $n \geq 2$ and $i \geq 1$, we have
  \[
    \pi_{i}(\Conf_n(\RR^2 \times \udisc^1)) \cong \bigoplus_{k = 0}^{n-1} \pi_{i}(\RR^2 \times \udisc^1 \setminus Q_k) \cong \bigoplus_{k = 1}^{n-1} \pi_i(\vee_k \usphere^{2}).
  \]
  In particular, $\Conf_{n}(\RR^2 \times \udisc^1)$ is simply connected.
\end{corollary}

\begin{proof}
  See \cite[Corollary~2.1]{FN62}.
\end{proof}

Now we are going to introduce a set of generators for $\pi_2(\Conf_n(\RR^{2} \times \udisc^1))$, which we will use in the computations of Section~\ref{sec:hssknotcal}.

\begin{definition}\label{def:genxij}
  For $1 \leq i \neq j\leq n$, define the map $x_{ij}$ as the composition of the following two maps
  \begin{align*}
    \usphere^2 &\to \Conf_{n-j+1}(\RR^{2} \times \udisc^1 \setminus Q_{j-1}) \\
    x &\mapsto (q_i + x, q_{j+1}, \dots, q_{n}),
  \end{align*}
  and
  \begin{align*}
    \Conf_{n-j+1}(\RR^{2} \times \udisc^1 \setminus Q_{j-1}) &\hookrightarrow \Conf_{n}(\RR^{2} \times \udisc^1)  \\
    (x_1, \dots, x_{n-j+1}) &\mapsto (q_1, \dots, q_{j-1}, x_1, \dots, x_{n-j+1}).
  \end{align*}
\end{definition}

\begin{proposition}\label{prop:genxij}
The maps $x_{ij} \colon \usphere^2 \to \Conf_n(\RR^2 \times \udisc^1)$ for $1 \leq i < j \leq n$ generate the group $\pi_{2}(\Conf_{n}(\RR^{2} \times \udisc^1))$. 
\end{proposition}

\begin{proof}
  The image $S_{ij} \coloneqq \im(x_{ij})$ of $x_{ij}$ is homeomorphic to a 2-sphere. 
  For a fixed $j$ with $1 \leq j \leq n$, the space $S_{1j} \vee S_{2j} \vee \cdots \vee S_{j-1, j} \subseteq \RR^2 \times \udisc^1$ is homotopy equivalent to the space~$\RR^2 \times \udisc^1 \setminus Q_{j-1}$.
  Note that for every $i$ with $1 \leq i < j$, the map $x_{ij}$ is the positive generator of $\pi_2(S_{ij})$.
  Thus, by Hurewicz isomorphism theorem, the maps $x_{ij}$ with~$i = 1, \dots, j-1$ generates the group $\pi_2(S_{1j} \vee S_{2j} \vee \cdots \vee S_{j-1, j}) \cong \pi_2(\RR^2 \times \udisc^1 \setminus Q_{j-1})$.
  Now let $j$ varies and apply Corollary~\ref{cor:piconfig}, we have that the maps $x_{ij}$ for $1 \leq i < j \leq n$ generate the group $\pi_{2}(\Conf_{n}(\RR^{2} \times \udisc^1))$. 
\end{proof}

\begin{remark}
 The proof provides a decomposition of $\pi_{\ast}(\Conf_n(\RR^{r} \times \udisc^1))$ as 
  \[
    \pi_{\ast}(\Conf_{n}(\RR^{r} \times \udisc^1)) \cong \bigoplus_{j=2}^{n}\pi_{\ast}(S_{1j} \vee \cdots \vee S_{ij} \vee \cdots \vee S_{j-1,j}),
  \]
  where for $1 \leq i < j \leq n$ the positive generator of $S_{ij}$ is $x_{ij}$.
  Thus by the following theorem of Hilton about the homotopy groups of wedges of spheres, we reduce the computation of homotopy groups of $\Conf_{n}(\RR^{2} \times \udisc^1)$ to homotopy groups of spheres.
\end{remark}

\begin{definition}[{\cite{Hil55},\cite[Page~511--512]{Whi78}}]\label{def:bp}
  Let $T \coloneqq \usphere^{r_1 +1} \vee \usphere^{r_2 +2} \vee \cdots \vee \usphere^{r_k +1}$ and denote by $\iota_{i}$ the positive generator of $\pi_{r_i+1}\left(\usphere^{r_i +1}\right)$.
  Note that $\iota_i$ can be considered as an element of $\pi_{r_i+1}(T)$ via the canonical embedding $\usphere^{r_i+1} \hookrightarrow T$.%
  \footnote{We consider basic products eventually as homotopy classes, but to get a well-defined definition, one has to first define basic products as `formal' products, \cf \cite[Page~511--512]{Whi78}. 
  The index set $P$ in Theorem~\ref{thm:hiltontheorem} below is then the set of formal basic products, and this ensures a posteriori that we do not have to distinguish formal products and Whitehead products of homotopy classes.}
  \begin{enumerate}
    \item The \emph{basic products} of weight 1 are the elements $\iota_1, \iota_2, \cdots , \iota_k$. We order the set of basic products of weight $1$ by $\iota_1 < \iota_2 < \cdots < \iota_k$. 
      We define basic products of weight bigger than $1$ recursively. 
      A basic product of \emph{weight $\omega$} is a Whitehead product $[a, b]$, where $a$ and $b$ are both basic products of weights $\alpha < \omega$ and $\beta < \omega$ respectively such that
      \begin{enumerate}[label=\alph*)]
        \item $\alpha + \beta = \omega$ and $a < b$, and 
        \item if $b$ is defined as the Whitehead product $[c,d]$ of basic products $c$ and $d$, then we have $c \leq a$.
      \end{enumerate} 
      We declare every basic product of weight $\omega$ to be greater than any basic product of smaller weight.
      We order the set of basic products of weight $\omega$ lexicographically, \ie for two basic products $[a,b]$ and $[a',b']$ of weight $\omega$, we set $[a,b] < [a',b']$ if $a < a'$, or $a = a'$ and $b < b'$.  
    \item Thus a basic product $p$ of weight $\omega$ is a suitably bracketed word in the symbols $\iota_i$ for $i = 1, \dots, k$. 
      Assume $\iota_i$ appears $w_i$ times in $p$.
      We define the \emph{height} $h(p)$ of $p$ as $\sum_{i = 1}^{k} r_i w_i$.
  \end{enumerate}
\end{definition}

\begin{theorem}[Hilton]\label{thm:hiltontheorem}
Using the notation of Definition~\ref{def:bp}, let $P$ be the set of (formal) basic products of $\iota_1, \dots, \iota_k$.
We have
\[
  \pi_{\ast}(T) \cong \bigoplus_{p \in P} \pi_{\ast}(\usphere^{h(p)+1}).
\]
where the direct summand $\pi_{\ast}(\usphere^{h(p)+1})$ is embedded in $\pi_{\ast}(T)$ by composition with the basic product $p \in \pi_{h(p)+1}(T)$.
\end{theorem}

\begin{proof}
  See \cite[Theorem~8.1]{Whi78}.
\end{proof}

The Whitehead products of $x_{ij}$ for $1 \leq i, j \leq n$ and $i \neq j$ of $\pi_2(\Conf_n(\RR^2 \times \udisc^1))$ satisfy some relations, which we will use in the computation in the Section~\ref{sec:hssknotcal}. 

\begin{proposition}[Hilton, Nakaoka--Toda, Massey--Uehara]\label{prop:whiteheadprod}
  Let $X$ be a topological space.
  Then the Whitehead product $[\blank, \blank]$ on $\pi_{\ast}(X)$ is bilinear, antisymmetric and satisfies the Jacobi identity, \ie for $\alpha \in \pi_{a+1}(X)$, $\beta \in \pi_{b+1}(X)$ and $\gamma \in \pi_{c+1}(X)$ we have
  \begin{enumerate}
    \item $[\alpha, \beta + \gamma] = [\alpha, \beta] + [\alpha, \gamma]$ and $[\alpha + \beta, \gamma] = [\alpha, \gamma] + [\beta, \gamma]$,
    \item $[\alpha, \beta] = (-1)^{(a+1)(b+1)}[\beta, \alpha]$, and 
    \item $(-1)^{c(a+1)}[\alpha, [\beta, \gamma]] + (-1)^{a(b+1)}[\beta, [\gamma, \alpha]] + (-1)^{b(c+1)}[\gamma, [\alpha, \beta]] = 0$.
  \end{enumerate}
\end{proposition}

\begin{proof}
  See \cite{Hil61}, \cite{MU57} and \cite{NT54}. 
\end{proof}

\begin{proposition}\label{prop:confrelation}
  Let $1 \leq i, j \leq n$ and $i \neq j$.
  The element $x_{ij} \in \pi_2(\Conf_n(\RR^2 \times \udisc^1))$ satisfy the following relations:
  \begin{enumerate}
    \item $x_{ij} = -x_{ji}$,
    \item $[x_{ij}, x_{jk}] = [x_{ji}, x_{ik}] = [x_{ik}, x_{kj}]$, if $n \geq 3$;
    \item $[x_{ij}, x_{kl}] = 0$, if $\{i,j\} \cap \{k, l\} = \emptyset$ and $n \geq 4$.
  \end{enumerate} 
\end{proposition}

\begin{proof}
  See the proof of \cite[Section II.3, Theorem 3.1]{FH01}. 
\end{proof}


\subsection{A homotopy spectral sequence for the Taylor tower of $\Emb(\blank)$}\label{sec:hssknotcal}
In this section we perform some computation of the integral homotopy Bousfield--Kan spectral sequence of cosimplicial space $\Embcos^{\bullet}$ (Definition \ref{def:cosmodel}) associated to the embedding functor $\Emb(\blank)$ (Notation~\ref{not:embfun}).
Recall that this spectral sequence aims at analysing the homotopy limit of the tower of fibrations
\[
\cdots \to \pT_n \Emb(\I) \to \pT_{n-1} \Emb(\I) \to \cdots \to \pT_0 \Emb(\I),
\]
according Corollary~\ref{cor:goodwilliecosprop}.

By Corollary~\ref{cor:goodwilliecosprop} we have the weakly homotopy equivalence
\begin{align}\label{eq:embcos&conf}
  \Embcos^{n} \simeq \Emb(V) \simeq \Conf_n(\RR^2 \times \udisc^1) \times \left(\usphere^2\right)^n,
\end{align}
for any $V \in \fbOpen{\I}^{\op}$ such that $\pi_{0}(\I \setminus V) \cong [n]$.
For the computation of the homotopy spectral sequence associated to $\Embcos^{\bullet}$ we need to compute the induced maps on homotopy groups of the coface and codegeneracy maps.
From (\ref{eq:embcos&conf}) we have 
\[
\pi_{\ast}(\Embcos^{n}) \cong \pi_{\ast}(\Conf_n(\RR^2 \times \udisc^1)) \times (\pi_{\ast}(\usphere^2))^n.
\]
By abuse of notation, we consider $x_{ij}$ (Proposition~\ref{prop:genxij}), for $1 \leq i < j \leq n$, as elements of $\pi_{\ast}(\Embcos^{n})$ under the natural inclusion.

Let $l \in \NN$ and $0 \leq l \leq n$.
Recall the notation from Theorem~\ref{thm:goodwilliecos}. 
Let $V_{n+1}$ be an open subset of $\I$ with $\partial \I \subseteq V_{n+1}$ such that $\kappa(V_{n+1}) = [n+1]$.
We obtain an open subset $V_{n} \subseteq V_{n+1}$ by removing the $(l+2)$-th subinterval of $V_{n+1} \setminus \partial \I$.
Then the codegeneracy map $s^{l}$ for $\Embcos^{\bullet}$ is the induced restriction map $\Emb(V_{n+1}) \to \Emb(V_{n})$, \ie forgetting the embedding of the $(l+2)$-th interval.
With respect to the homotopy equivalence \ref{eq:embcos&conf}, we can write $s^{l}$ with $0 \leq l \leq n$ concretely as
\begin{align*}
  s^l \colon \Conf_{n+1}(\RR^2 \times \udisc^1) \times \left(\usphere^2\right)^{n+1} &\to \Conf_{n}(\RR^2 \times \udisc^1) \times \left(\usphere^2\right)^{n}\\
  (x_1, \dots, x_{n+1}) \times (v_1, \dots, v_{n+1}) &\mapsto (x_1, \dots, \widehat{x_{l+1}}, \dots, x_{n+1}) \times (v_1, \dots, \widehat{v_{l+1}}, \dots, v_{n+1}),
\end{align*}
where $\widehat{(\blank)}$ denotes taking out the elements.

\begin{convention}\label{conv:config-tangent}
  Denote by $s_{\ast}^l(c) \colon \pi_{r}(\Conf_{n+1}(\RR^2 \times \udisc^1)) \to \pi_r(\Conf_{n}(\RR^2 \times \udisc^1))$ the restriction of the map $s_{\ast}^{l}$ to the $\pi_{r}(\Conf_{n+1}(\RR^2 \times \udisc^1))$ component.
  Denote by $s_{\ast}^l(t)$ the restriction $\left(\pi_{r}(\usphere^2)\right)^{n+1} \to \left(\pi_{r}(\usphere^2)\right)^{n}$ of the map $s_{\ast}^{l}$ on the $\left(\pi_{r}(\usphere^2)\right)^{n+1}$ component.
\end{convention}

\begin{proposition}\label{prop:degeneracy-map}
  Let $i, j, l, n \in \NN$ and $1 \leq i < j \leq n+1$ and $0 \leq l \leq n$ and $n \geq 2$.
  \begin{enumerate}
    \item We have
      \[
        s^l_{\ast}\left(x_{ij}\right) =
        \begin{cases}
          x_{i-1, j-1} &\text{if } l < i-1\\
          x_{i, j-1} &\text{if } i-1 < l < j-1 \\
          x_{i, j} &\text{if } l > j-1 \\
          0 &\text{otherwise.}
        \end{cases}
      \]
    \item Denote by $Z$ the set of basic products of the elements $x_{i,j}$ containing $x_{u,l+1}$ or~$x_{l+1, v}$ for $1 \leq u \leq l$ and $l+2 \leq v \leq n+1$.
      Under the isomorphism in Theorem~\ref{thm:hiltontheorem}, the kernel of the map $s_{\ast}^l(c)$ is isomorphic to $\bigoplus_{p \in Z} \pi_{r}(\usphere^{h(p)}+1)$, for $r \geq 2$.  
    \item For $r \geq 2$ , the map $s_{\ast}^l(t)$ is the canonical projection where forgetting the $l$-th component.
      Thus the kernel of $s_{\ast}^{l}(t)$ is isomorphic to $\left(0\right)^{l-1} \times \pi_{r}(\usphere^2) \times \left(0\right)^{n-l}$.
  \end{enumerate}
\end{proposition}

\begin{proof}
  \rom{1} and \rom{3} follows from the description of $s^l$ right above the proposition.
  
  For the proof of \rom{2}, let us abbreviate $s_{\ast}^{l}(c)$ by $s_{\ast}^{l}$ in this part of the proof.
  Note that for~$n \geq 2$, we have $s_{\ast}^l (x_{uv}) = 0$ if and only if $u = l+1 \text{ or } v = l+1$. 
  Therefore, for $n \geq 2$ and $z \in Z$, we have $s_{\ast}^{l}(z) = 0$ by the naturality of the Whitehead product.
  Thus $s_{\ast}^{l}$ factors through $\quot{\pi_r (\Conf_n(\RR^2 \times \udisc^1))}{{\bigoplus_{p \in Z} \pi_{r}(\usphere^{h_(p)}+1)}}$
  \[
  \begin{tikzcd}[row sep = huge]
    \pi_{r}(\Conf_{n+1}(\RR^2 \times \udisc^1)) \arrow[r, "s_{\ast}^l"] \arrow[d, "p"'] & \pi_{r}(\Conf_{n}(\RR^2 \times \udisc^1)) \\
    \quot{\pi_r (\Conf_n(\RR^2 \times \udisc^1))}{{\oplus_{p \in Z} \pi_{r}(\usphere^{h(p)}+1)}} \arrow[ru, "\bar{s}_{\ast}^l"', dashed] & 
  \end{tikzcd}
  \]
  
  By inspecting the value of $s_{\ast}^l$ on $x_{ij}$, we conclude that for two basic products $w_1 \leq w_2$ with $s_{\ast}^l(w_k) \neq 0$ for $k \in \{1, 2\}$, we have $s_{\ast}^l(w_1) \leq s_{\ast}^l(w_2)$. 
  Also the heights of $w_k$ and $s^{l}_{\ast}(w_k)$ are the same. 
  
  So $\bar{s}_{\ast}^l$ sends a basis of $\quot{\pi_r (\Conf_n(\RR^2 \times \udisc^1))}{{\sum_{p \in Z} \pi_{r}(\usphere^{h(p)+1})}}$ (\ie basic products that are not in $Z$) injectively to a basis of $\pi_{r}(\Conf_{n}(\RR^2 \times \udisc^1))$. 
  Thus $\bar{s}_{\ast}^l$ is injective, which implies that the kernel of $s_{\ast}^l$ is isomorphic to $\bigoplus_{p \in Z} \pi_{r}(\usphere^{h(p)+1})$, via the isomorphism from Theorem~\ref{thm:hiltontheorem}.
\end{proof}

Similar analysis of the definition of the coface maps tells us that these maps $\delta^{l}$ of $\Embcos^{\bullet}$ corresponds to ``breaking'' the embeddings of the $(l+1)$-th interval into the embedding of two subintervals.
 Therefore, one representative for the map $\delta^l$ with $0 < l < n+1$ is the following: 
\begin{align*}
  \Conf_n(\RR^2 \times \udisc^1) \times \left(\usphere^2\right)^{n} &\to \Conf_{n+1}(\RR^2 \times \udisc^1) \times \left(\usphere^2\right)^{n+1} \\
  (x_1, \dots, x_n) \times (v_1, \dots, v_n) &\mapsto \\
  & (x_1, \dots, x_l, x_{l}+\epsilon v_l, x_{l+1}, \dots, x_n) \times (v_1, \dots, v_l, v_l, v_{l+1}, \dots, v_n)
\end{align*}
where the scalar $\epsilon \in \RR$ is so chosen that $(x_1, \dots, x_l, x_{l}+\epsilon v_l, x_{l+1}, \dots, x_n)$ is a well-defined point in $\Conf_{n+1}(\RR^2 \times \udisc^1)$.
For $l = 0$ and $l = n+1$, we have
\begin{align*}
  \delta^0\left((x_1, \dots, x_n) \times (v_1, \dots, v_n)\right) &= (x_{-1}+\epsilon e, x_1, \dots, x_n) \times (e, v_1, \dots, v_n) \\
   \delta^{n+1}\left((x_1, \dots, x_n) \times (v_1, \dots, v_n)\right) &= (x_1, \dots, x_n, x_{+1}+\epsilon' e,) \times (v_1, \dots, v_n, e),
\end{align*}
where $x_{-1} = (0, 0, -1)$ and $x_{+1} = (0, 0, 1)$ and $e = (0, 0, 1)$.

By explicit calculation we obtain the following  
  
\begin{proposition}\label{prop:compmaps}
Let $i, j, l, n \in \NN$ such that $n \geq 2$ and $1 \leq i < j \leq n$ and $0 \leq l \leq n+1$. 
\begin{enumerate}
  \item For $n \in \NN$ and $n \geq 2$, we have 
      \[
        \delta^l_{\ast}\left(x_{ij}\right) =
        \begin{cases}
          x_{i+1, j+1} &\text{if } l < i \\
          x_{i, j+1} + x_{i+1, j+1} &\text{if } l = i \\
          x_{i, j+1} &\text{if } i < l < j \\
          x_{i, j}+x_{i, j+1}  &\text{if } l = j \\
          x_{ij} &\text{otherwise.}  
        \end{cases}
      \]
  \item Denote by $y_k$ a generator for the $k$-th component $\pi_{2}(\usphere^{2})$ of $(\pi_{2}(\usphere^2))^n$.
    We have that 
      \[
      \delta_{\ast}^l(y_k) = 
      \begin{cases}
        y_{k+1} &\text{if } l < k \\
        x_{k, k+1} + y_k + y_{k+1} &\text{if } l = k \\
        y_k &\text{otherwise.} 
      \end{cases}
      \]
\end{enumerate} 
     \qed
\end{proposition}

Now we can compute $E_{p-1, p}^{1}$ and $E_{p,p}^1$ in the homotopy spectral sequence associated to the cosimplicial space $\Embcos^{\bullet}$.

\begin{corollary}\label{cor: essentialker}
  Let $l, n, r \in \NN$ and $n \geq 2$ and $0 \leq l \leq n$ and $r \geq 2$, and recall the notations from Convention~\ref{conv:config-tangent}.
  For the degeneracy map $\Embcos^{n+1} \xrightarrow{s^l} \Embcos^{n}$, we have 
  \begin{align*}
    \ker s_{\ast}^l &= \ker s_{\ast}^l(c) \times \ker s_{\ast}^l(t) \text{ and} \\
    \bigcap_{l = 0}^{n-1} \ker s_{\ast}^l  &\cong  \bigcap_{l = 0}^{n-1} \ker s_{\ast}^l(c) \times \left(0\right)^{n}.
  \end{align*}
\end{corollary}

\begin{proof}
  We have that $s_{\ast}^l = s_{\ast}^l(c) \times s_{\ast}^{l}(t)$.
\end{proof}

\begin{proposition}\label{prop:e1page}
  \begin{enumerate}
    \item For $p \geq 3$ and $1 \leq i < p-1$, let $T$ be the set of basic products of the elements~$x_{i, p-1}$ of height $p-2$, such that each $x_{i, p-1}$ appears exactly once. Let $F$ be the set of basic products of elements $x_{i, p-1}$ of height $p-1$, such that one $x_{k,p-1}$ appears exactly twice and all other $x_{i,p-1}$ appear exactly once. 
      Then we have
    \begin{equation}\label{eq:ep-1p}
       E_{p-1,p}^1 \cong \bigoplus_{T}\pi_{p}(\usphere^{p-1}) \oplus \bigoplus_{F}\pi_{p}(\usphere^{p})
    \end{equation}
    where $\pi_p(\usphere^{p-1})$ and $\pi_{p}(\usphere^p)$ are embedded in $\pi_p(\Conf_{p-1}(\RR^2 \times \udisc^1))$ by composition with the basic products in $T$ and $F$ respectively. 
    \item For $p \geq 2$, let $H$ be the set of basic products of height $p-1$ of the elements in $x_{i,p}$ for $1 \leq i \leq p-1$ such that each $x_{i,p}$ appears exactly once. 
      Then 
    \begin{equation}\label{eq:epp}
      E_{p,p}^1 \cong \bigoplus_{H}\pi_{p}(\usphere^p),
    \end{equation}
    where the direct summands $\pi_p(\usphere^p)$ are embedded in $\pi_{p}(\Conf_{p}(\RR^2 \times \udisc^1))$ by composition with the basic products in $H$.
  \end{enumerate}
\end{proposition}

\begin{proof}
 \rom{1} Recall from Proposition~\ref{prop:cosimplicialss} that $E_{p-1, p}^1 \cong \pi_p(\Embcos^{p-1}) \cap \bigcap_{l = 0}^{p-2} \ker(s^l_{\ast})$.
    By Corollary~\ref{cor: essentialker} we only need to consider the $\pi_{p}(\Conf_{p-1}(\RR^2 \times \udisc^1))$ component of $\pi_p(\Embcos^{p-1})$.
    In other words, 
    \[
    E_{p-1, p}^{1} \cong \pi_{p}(\Conf_{p-1}(\RR^2 \times \udisc^1)) \cap \bigcap_{l = 0}^{p-2} \ker(s^l_{\ast}(c)).
    \]
    Recall from Corollary~\ref{cor:piconfig} that 
    \[
      \pi_{p}(\Conf_{p-1}(\RR^2 \times \udisc^1)) \cong \bigoplus_{j = 2}^{p-1} \pi_{p}(S_{1j} \vee S_{2j} \vee \cdots \vee S_{j-1,j}),
    \] 
    and $x_{ij}$ is the positive generator of $S_{ij}$, $1 \leq i < j \leq p-1$. 
    For a fixed $j$, let $\{b^{(j)}_k\}_{k \in \NN}$ be the set of basic products of the elements $x_{ij}$ for $i = 1, \dots, j-1$.
    Using Theorem~\ref{thm:hiltontheorem}, we have
    \[
      \pi_{p}(\Conf_{p-1}(\RR^2 \times \udisc^1)) \cong \bigoplus_{\mathclap{\substack{1 < j < p-1 \\ k \in \NN, \ h ( b^{(j)}_k ) \leq p-1}}}  \pi_p(\usphere^{h ( b^{(j)}_k ) + 1}).
    \]    
    Next we need to examine which elements of $\pi_{p}(\Conf_{p-1}(\RR^2 \times \udisc^1))$ lie in $\bigcap_{l = 1}^{p-2} \ker(s^l_{\ast}(c))$. 
    By Proposition~\ref{prop:degeneracy-map}.\rom{3}, it is sufficient to see which basic products lie in $\bigcap_{i = 1}^{p-2} \ker(s^l_{\ast}(c))$.
    Let us consider the following cases:
    \begin{enumerate}[label = \alph*)]
      \item Suppose $j \neq p-1$. 
        In this case, we have $s^{p-1}_{\ast}(b^{(j)}_{k}) \neq 0$.
      \item Suppose $j = p-1$ and $h(b^{(j)}_k) \leq p-3$.
        In this case there exists at least one index $1 \leq i < p-1$ such that $x_{i,p-1}$ does not appear in $b^{(j)}_{k}$, and thus $s^{i-1}_{\ast}(b^{(j)}_k) \neq 0$.
      \item Suppose $j = p-1$ and $h(b^{(j)}_k) = p-2$.
        In this case each $x_{i,p-1}$ with $1 \leq i \leq p-2$ appears in $b^{(j)}_{k}$ exactly once.
        Thus for all $0 \leq l \leq p-2$, we have $s^{l}_{\ast}(b^{(j)}_k) = 0$, since $s^{l}_{\ast}(x_{l+1,p-1}) = 0$ and $x_{l+1, p-1}$ appears in $b^{(j)}_k$.
      \item Suppose $j = p-1$ and $h(b^{(j)}_k) = p-1$.
        In this case there exists an index $i_k$ such that $x_{i_k, p-1}$ appears exactly twice in $b^{(j)}_k$, and all other $x_{i,p-1}$ with $1 \leq i \leq p-2$ and $i \neq i_k$ appear exactly once in $b^{(j)}_k$.
        As in c) we see $s_{\ast}^l(b^{(j)}_k) = 0$ for $0 \leq l \leq p-2$.
    \end{enumerate}
    Thus $\bigcap_{i = 1}^{p-2} \ker(s^i_{\ast}(c))$, or $E_{p-1, p}^1$, is generated by basic products of the form in c) and d), which yields (\ref{eq:ep-1p}).
    
 \rom{2} Similar to \rom{1}, we recall that $E_{p, p}^1 \cong \pi_{p}(\Conf_{p}(\RR^2 \times \udisc^1)) \cap \bigcap_{l = 0}^{p-1} \ker(s^{l}_{\ast}(c))$. 
    Furthermore $\bigcap_{l = 0}^{p-1} \ker(s^{l}_{\ast}(c))$ is generated by the basic products $\{b^{(p)}_k\}_{k \in \NN}$ of elements $x_{ip}$ with~$1 \leq i \leq p-1$ such that $h(b^{(p)}_k) = p-1$, and for each $i$ with $1 \leq i \leq p-1$, the element~$x_{ip}$ appears exactly once in $b^{(p)}_k$.
\end{proof}

\begin{remark}\label{rmk:e1page}
  Using the same method as in the proof above, we can see that the abelian group $E_{p, q}^{1}$ for $0 \leq p \leq q$ is isomorphic to the direct sum of the $q$-th homotopy groups of spheres of dimension at most $q$, indexed by the basic products where all elements $\{x_{i, p} \mid 1 \leq i \leq p-1\}$ appear at least once\footnote{The dimension of the spheres and the occurrences of the elements are required to be compatible with each other.}. 
\end{remark}

\begin{remark}\label{rmk:e1BCKS}
  Budney--Conant--Koytcheff--Sinha gives the general formula for the abelian groups $E_{p, q}^{1}$, \cf \cite[Proposition 7.2]{BCKS17} by computing the Bousfield--Kan spectral sequence with integer coefficients associated to the cosimplicial model $C^{\bullet}$ (Remark~\ref{rmk:othermodel}).
  The results in the above proposition agree with those from \cite[Proposition 7.2]{BCKS17} once one removes the homotopy groups of spheres that are 0.
  Our proof, using another cosimplicial model $\Embcos^{\bullet}$ and doing the computation directly from definitions, provides an alternative approach of the computation of the spectral sequences, as well as more details for the arguments of \cite[Proposition 7.2]{BCKS17}.
\end{remark}

With the description of $E_{p-1,p}^1$ and $E_{p,p}^1$ in terms of elements $x_{ij}$ where $1 \leq i \leq j$ and~$j = p-1$ or $j = p$, we are going to give an explicit and simplified formula for the differential $d^1\colon E_{p-1,p}^1 \to E_{p,p}^1$ now.

\begin{proposition}\label{prop:d1lowdim}
  \begin{enumerate}
    \item The differential $d^1 \colon E_{p-1,p}^1 \to E_{p,p}^1$ is trivial for $p = 1$ and $p = 3$.
    \item For $p = 2$, the differential $d^1 \colon E_{1, 2} \to E_{2, 2}$ is an isomorphism.
  \end{enumerate} 
\end{proposition}

\begin{proof}
  First we have 
  \begin{align*}
    E_{0,1}^1 &\cong \pi_1(\Embcos_{0}) = 0 \text{, and} \\
    E_{1,1}^1 &\cong \pi_1(\Embcos_{1}) \cong \pi_1(\Conf_{1}(\RR^2 \times \udisc^1))  \times \pi_1(\usphere^2) = 0. 
  \end{align*}
  Thus $d^1 \colon E_{0,1}^1 \to E_{1,1}^1$ is trivial.
  
  For $d^1\colon E_{2,3}^1 \to E_{3,3}^1$, we apply Proposition~\ref{prop:e1page} to see that
  \[
    E_{2,3}^1 \cong \pi_3(S^2) \cong \ZZ,
  \]
  with generator $x_{12}$. 
  Thus
  \begin{align*}
    d^1(x_{12}) &= \sum_{i = 0}^{3}\delta_{\ast}^i(x_{12}) \\
                &= x_{23} - (x_{13}+x_{23})+(x_{12}+x_{13})-x_{12}\\
                & = 0.
  \end{align*}
  Therefore $d^1 \colon E_{2,3}^1 \to E_{3,3}^1$ is trivial.
  
  As for $d^1 \colon E_{1,2}^1 \to E_{2,2}^1$, we have 
  \[
    E_{1,2}^1 \cong \pi_2(\Embcos^{1}) \cong \pi_2(\Conf_{1}(\RR^2 \times \udisc^1)) \times \pi_2(\usphere^2) \cong \{0\} \times \pi_2(\usphere^2)
  \]
  with generator $y_1$ for the component $\pi_2(\usphere^2)$, and
  \[
    E_{2, 2}^1 \cong \pi_2(\Embcos^{2}) \cap \bigcap_{l = 0}^{1} s_{\ast}^l \cong \pi_{2}(\Conf_2(\RR^2 \times \udisc^1)) \times \{e\} \times \{e\},
  \]
  with generator $x_{12}$ for the component $\pi_{2}(\Conf_2(\RR^2 \times \udisc^1))$.
  Applying Proposition~\ref{prop:cosimplicialss}, we have
  \begin{align*}
    d^1(y_1) &= \sum_{i = 0}^{2}\delta_{\ast}^i(y_1) \\
                &= y_2-(x_{12}+y_1+y_2)+y_1\\
                & = x_{12}.
  \end{align*}
  Therefore, $d^1 \colon E_{1,2}^1 \to E_{2,2}^1$ is an isomorphism.
\end{proof}

Now let us consider $d^1 \colon E_{p-1, p}^1 \to E_{p,p}^1$, for $p \geq 4$.
According to Proposition~\ref{prop:e1page} we have for $p \geq 4$,
\begin{align*}
  E_{p-1,p}^1 &\cong \bigoplus_{T}\pi_{p}(\usphere^{p-1}) \oplus \bigoplus_{F}\pi_{p}(\usphere^p) \\
              & \cong \bigoplus_{T}\ZZ/2\ZZ \oplus \bigoplus_{F}\ZZ \\
  E_{p,p}^1 &\cong \bigoplus_{H}\pi_p(\usphere^p) \cong \bigoplus_{H}\ZZ.
\end{align*}
Since $E_{p,p}^1$ is torsion free, we see that $d^1$ is trivial on $\bigoplus_{T}\pi_p(\usphere^{p-1})$, and we conclude that we only need to consider the restriction of $d^1$ to $\bigoplus_{F}\pi_p(\usphere^p)$. 

\begin{notation}\label{not:fep-1p}
  We denote the torsion-free part of $E_{p-1,p}^1$ by $\quot{E_{p-1,p}^1}{\mathrm{tors}}$, \ie the summand $\bigoplus_{F}\pi_p(\usphere^p)$ in Equation~\ref{eq:ep-1p}.
\end{notation}

\begin{proposition}\label{prop:separateindices}
  For $p \geq 4$, denote by $D_{p}^{\mathrm{sep}}$ the set of iterated Whitehead products of the elements $x_{i,p-1}$ for $i = 1, \dots, p-2$ with the following properties:
  \begin{enumerate}
    \item For every $w \in D_{p}^{sep}$, there exists one $x_{k(w),p-1}$ that appears exactly twice and all other $x_{i,p-1}$ with $1 \leq i \leq p-2$ and $i \neq k(w)$ appear exactly once.
    \item Every $w \in D_{p}^{sep}$ is of the form $w = [c_1, c_2]$ where $c_1$ is an iterated Whitehead product of elements $x_{i,p-1}$ with $i \in I$ and $c_2$ is an iterated Whitehead product of elements $x_{i,p-1}$ with $j \in J$ such that $I, J \subseteq \{1, \dots, p-2\}$, $I \cap J = \{k(w)\}$ and $I \cup J = \{1, 2, \dots, p-2\}$.
  \end{enumerate}
  Then, $\quot{E_{p-1,p}^1}{\mathrm{tors}}$ is generated by $D_{p}^{\mathrm{sep}}$.
\end{proposition}

\begin{proof}
  Denote by $D_{p}$ the set of iterated Whitehead products of the elements $x_{i,p-1}$ with $i = 1, \dots, p-2$ satisfying only condition \rom{1}.
  Using the same argument as in the proof of Proposition~\ref{prop:e1page}, we see that $D_{p}^{\mathrm{sep}} \subseteq E_{p-1, p}^1$ and $D_{p} \subseteq E_{p-1,p}^1$.
  In particular, the basic products in $F$, which are a basis of $\quot{E_{p-1,p}^1}{\mathrm{tors}}$, are contained in $D_{p}$.   
  We have reduced the desired statement to the following claim which we prove by induction.
 
 \textbf{Claim.}
   \textit{For $p \geq 4$, any element of $D_p$ can be written as a linear combination of elements of $D_{p}^{\mathrm{sep}}$ using only the Jacobi identity and antisymmetry relations (Proposition~\ref{prop:whiteheadprod}).}

 For $p = 4$, the claim follows by listing all the elements of $D_4$ and using the Jacobi identity of the Whitehead product.
 
 Assume that the claim is true for all $p \leq n$ with $n \geq 4$.
 Let $p = n+1$ and consider~$\widetilde{w} = [a_1, a_2]\in D_{n+1}$. Without loss of generality, we can assume that $x_{1, n}$ is the repeated element in $\widetilde{w}$.
 If the two copies of $x_{1,n}$ appear in $a_1$ and $a_2$ separately, then $\widetilde{w}$ is already an element of $D_{p}^{\mathrm{sep}}$. 
 Otherwise, both copies of $x_{1,n}$ appear in either $a_1$ or $a_2$, say they appear in $a_1$. 
 By assumption $a_1$ is a Whitehead product of elements $x_{m,n}$.
 Define the set $M \coloneqq \{m \in \NN | x_{m, n} \text{ appears in } a_1\}$.
 So we know $\# M \leq n-2$ and $1 \in M$. 
 The element $x_{m, n}$ appears exactly once in $a_1$, for $1 \neq m \in M$
 
 There is a bijection $r \colon \{x_{m,n} \mid m \in M\} \to \{x_{i, \#M +1} \mid i = 1, \dots, \#M\}$ such that $x_{1,n}$ is mapped to $x_{1,\#{M}+1}$. 
 Define $a_1' \in D_{\#{M}+2}$ by replacing each occurrence of $x_{m,n}$ in $a_1$ by~$r(x_{m,n})$. 
 By an inductive assumption, we can write $a_1'$ as a finite sum $a_1' = \sum_{i \in I} [c_{i1}', c_{i2}']$  such that $[c_{i1}', c_{i2}'] \in D_{\#{M}+2}^{\mathrm{sep}}$ and $x_{1,\#{M}+1}$ appears exactly twice in $[c_{i1}, c_{i2}]$. 
 Thus, by replacing each $x_{i, \# M +1}$ by $r^{-1}(x_{i, \# M +1})$ we obtain $a_1 = \sum_{i \in I} [c_{i1}, c_{i2}]$ such that $[c_{i1}, c_{i2}]$ is a Whitehead product of the elements $x_{m,n}$ with $m \in M$, where $x_{1,n}$ appears exactly twice and $x_{m,n}$ appears exactly once for $m \neq 1$.
 
 Therefore $\widetilde{w}$ can be written as
 \begin{align*}
   \widetilde{w} &= \sum_{i \in I}\big[[c_{i1}, c_{i2}], a_2\big] \\
             &= \sum_{i \in I}(-1)^{\epsilon_1}\big[[c_{i1}, a_2], c_{i2}\big] + (-1)^{\epsilon_2}\big[[a_2, c_{i2}], c_{i1}\big],
 \end{align*}
 where $\epsilon_1$ and $\epsilon_2$ denote the signs which come from the Jacobi identity for the Whitehead product.
 For every $i \in I$, we have that $\left[[c_{i1}, a_2], c_{i2}\right], \left[[a_2, c_{i2}], c_{i1}\right] \in D_{n+1}^{\mathrm{sep}}$. 
 Thus $\widetilde{w}$ is a linear combination of elements of $D_{p}^{\mathrm{sep}}$.
\end{proof}

The upshot is that it is sufficient, for the computation $d^1 \colon E_{p-1,p}^1 \to E_{p,p}^1$, to compute its evaluation on $w$ for every $w \in D_{p}^{\mathrm{sep}}$.

\begin{theorem}\label{prop:d1simplified}
  Let $w = [c_1, c_2] \in D_{p}^{\mathrm{sep}}$, say with repeated occurrence of $x_{k,p-1}$. 
  We write\footnote{By abuse of notation we hide the inner brackets of iterated Whitehead products when we write an elements as done here.} $c_1$ and $c_2$ as $c_1 = [\dots x_{i,p-1}\dots x_{k,p-1}\dots]$ and $c_2=[\dots x_{k,p-1}\dots x_{j,p-1}\dots]$.
  The index $k$ will be fixed through out the proposition.
  Then we have 
  \[
    d^1(w) = \partial^k(w) + \partial^{p-1}(w),
  \]
  where
  \begin{multline*}
    \partial^k(w)= (-1)^k \big[ [\dots x_{i', p}\dots x_{k,p}\dots], [\dots x_{k+1,p}\dots x_{j', p}\dots] \big] \\+(-1)^k \big[ [\dots x_{i', p}\dots x_{k+1,p}\dots],[\dots x_{k,p}\dots x_{j', p}\dots] \big]
  \end{multline*}
  and
  \begin{multline*}
    \partial^{p-1}(w) = (-1)^{p-1} \big[[\dots x_{i,p-1}\dots x_{k,p-1}\dots], [\dots x_{k,p}\dots x_{j,p}\dots] \big] \\
    +(-1)^{p-1} \big[[\dots x_{i, p}\dots x_{k,p}\dots], [\dots x_{k,p-1}\dots x_{j,p-1}\dots]\big],
  \end{multline*}
  where $i'=i$ if $i < k$ and $i'=i+1$ if $i > k$ and $j'=j$ if $j < k$ and $j'=j+1$ if $j > k$.
\end{theorem}

Before we prove the theorem, let us take a look at an example of computation of~$d^1$.

\begin{example}
  Let $k = 2$ and $p = 8$ and $w \in D^{\mathrm{sep}}_8$ of the form
  \[
    w = \big[ \left[ [x_{37}, x_{27}], x_{57} \right], \left[ [x_{17}, x_{27}], [x_{47}, x_{67}]\right] \big].
  \]
  Now let us calculate $d^1(w)$ using the formulas in Theorem~\ref{prop:d1simplified}.
  We have 
  \[
    \partial^2(w) = \big[ \left[ [x_{48}, x_{28}], x_{68} \right], \left[ [x_{18}, x_{38}], [x_{58}, x_{78}]\right] \big] + \big[ \left[ [x_{48}, x_{38}], x_{68} \right], \left[ [x_{18}, x_{28}], [x_{58}, x_{78}]\right] \big],
  \] 
  and
  \[
    \partial^{7}(w) = -\big[ \left[ [x_{37}, x_{27}], x_{57} \right], \left[ [x_{18}, x_{28}], [x_{48}, x_{68}]\right] \big] -\big[ \left[ [x_{38}, x_{28}], x_{58} \right], \left[ [x_{17}, x_{27}], [x_{47}, x_{67}]\right] \big].
  \]
\end{example}

\begin{proof}[Proof of Theorem \ref{prop:d1simplified}] First we proof the proposition for $k \leq p-3$.

  \textbf{Claim 1.}
    \textit{For $l \neq k$, $p-1, p$, each of the elements $(-1)^{l-1}\delta_{\ast}^{l-1}(w)$, $(-1)^{l}\delta_{\ast}^{l}(w)$ and $(-1)^{l+1}\delta_{\ast}^{l+1}(w)$} can be written canonically as a sum of two iterated Whitehead products of $x_{i,p}$ with $1 \leq i < p$ such that every summand of $(-1)^{l}\delta_{\ast}^{l}(w)$ appears in $(-1)^{l-1}\delta_{\ast}^{l-1}(w)$ or $(-1)^{l+1}\delta_{\ast}^{l+1}(w)$ with opposite sign.
    
    We note first that $\delta_{\ast}^{l-1}(x_{i, p-1}) = \delta_{\ast}^{l}(x_{i, p-1}) = \delta_{\ast}^{l+1}(x_{i, p-1})$ for $i \neq l-1$, $l$ and $l+1$. 
    Moreover,
    \[
      \delta_{\ast}^{r}(x_{i, p-1})=
      \begin{cases}
        x_{i+1, p}, &\text{if } i > l+1 \\
        x_{i,p}, &\text{if } i < l-1
      \end{cases}
    \]
    for $r = l-1$, $l$ and $l+1$.

    Without loss of generality, we write $w = [\dots x_{l,p-1}\dots x_{l-1,p-1}\dots x_{l+1,p-1}\dots]$, where we show only the elements that are interesting for us. 
    Note that in this presentation of $w$, the order in which the elements $x_{l,p-1}$,  $x_{l-1,p-1}$ and $x_{l+1,p-1}$ appear does not play a role. 
    We calculate $\delta^{r}_{\ast}(w)$ for $r = l-1, l$ and $l+1$:
    \begin{align*}
      \delta_{\ast}^{l-1}(w) &= [\dots x_{l+1, p}\dots x_{l-1, p} + x_{l, p}\dots x_{l+2, p}\dots] \\
      \delta_{\ast}^{l}(w) &= [\dots x_{l,p} + x_{l+1, p}\dots x_{l-1, p}\dots x_{l+2, p}\dots] \\
      \delta_{\ast}^{l+1}(w) &= [\dots x_{l, p}\dots x_{l-1, p}\dots x_{l+1, p}+x_{l+2, p}\dots]
    \end{align*}
    Thus we see that $[\dots x_{l+1, p}\dots x_{l-1, p}\dots x_{l+2, p}\dots]$ of $\delta_{\ast}^{l}(w)$ appears in $\delta_{\ast}^{l-1}(w)$ and the term $[\dots x_{l, p}\dots x_{l-1, p}\dots x_{l+2, p}\dots]$ of $\delta_{\ast}^{l}(w)$ appears in $\delta_{\ast}^{l+1}(w)$.
    Since the signs in front of $\delta_{\ast}^{r}$ with $r = l-1$, $l$ and $l+1$ in $d^1$ alternate, we see that in $d^1(w)$ the terms of $(-1)^{l}\delta_{\ast}^l(c)$ are cancelled by terms of $(-1)^{l-1}\delta_{\ast}^{l-1}(w)$ and $(-1)^{l+1}\delta_{\ast}^{l+1}(w)$ as desired.
    
  \textbf{Claim 2.}
    \textit{For $l = k$, after cancelling terms of $(-1)^k \delta_{\ast}^k(w)$ by terms of $(-1)^{k-1}\delta_{\ast}^{k-1}(w)$ and $(-1)^{k+1}\delta_{\ast}^{k+1}(w)$ as in \textbf{Claim 1.}, the remaining terms of $(-1)^k \delta_{\ast}^k(w)$ is equal to $\partial^k(w)$.}
    
    For convenience of the proof, we write without loss of generality
    \[
      w = [\dots x_{k,p-1}\dots x_{k-1,p-1}\dots x_{k,p-1}\dots x_{k+1,p-1}\dots].
    \] 
    Again note that for $i \neq k-1$, $k$, $k+1$, we have $\delta_{\ast}^{k-1}(x_{i,p-1}) = \delta_{\ast}^k(x_{i,p-1}) = \delta_{\ast}^{k+1}(x_{i,p-1})$, and note
    \begin{align*}
      \delta_{\ast}^{k-1}(w) &= [\dots x_{k+1, p}\dots x_{k-1, p}+x_{k, p}\dots x_{k+1, p}\dots x_{k+2, p}\dots] \\
      \delta_{\ast}^k(c) &= [\dots x_{k,p}+x_{k+1, p}\dots x_{k-1, p}\dots x_{k, p}+x_{k+1, p}\dots x_{k+2, p}\dots] \\
      \delta_{\ast}^{k+1}(w) &= [\dots x_{k,p}\dots x_{k-1,p}\dots x_{k, p}\dots x_{k+1, p}+x_{k+2, p}\dots].
    \end{align*}
    Thus after cancelling with terms of $\delta_{\ast}^{k-1}(w)$ and $\delta_{\ast}^{k+1}(w)$, the remaining term of $\delta_{\ast}^k(w)$ is
    \begin{multline*}
      (-1)^k[\dots x_{k,p}\dots x_{k-1,p}\dots x_{k+1, p}\dots x_{k+2, p} \dots] \\+ (-1)^k[\dots x_{k+1,p}\dots x_{k-1,p}\dots x_{k,p}\dots x_{k+2, p}\dots].
    \end{multline*}
    Writing $w$ as $w=[c_1, c_2] = \left[[\dots x_{i', p-1}\dots x_{k, p-1} \dots], [\dots x_{k, p-1}\dots x_{j', p-1}\dots] \right]$, the remaining terms of $\delta_{\ast}^k(w)$ look like
    \begin{multline*}
      (-1)^k\big[[\dots x_{i', p}\dots x_{k,p}\dots], [\dots x_{k+1, p}\dots x_{j', p}\dots]\big] \\ + (-1)^k\big[[\dots x_{i',p}\dots x_{k+1, p}\dots], [\dots x_{k,p}\dots x_{j', p}\dots]\big],
    \end{multline*}
    which we recognise as $\partial^k(w)$ as desired.

  \textbf{Claim 3.}
    After cancelling with terms of $(-1)^{p-2}\delta_{\ast}^{p-2}(w)$ and $(-1)^{p}\delta_{\ast}^{p}(w)$, the remaining term of $(-1)^{p-1}\delta_{\ast}^{p-1}(w)$ is $\partial^{p-1}(w)$.
    
    In order to prove this claim, we write $w = [c_1, c_2]$ as in Proposition~\ref{prop:d1simplified}. 
    Recall that~$\delta_{\ast}^{p-1}(x_{i,p-1}) = x_{i,p-1} + x_{i,p}$, so we have
    \begin{equation*}
      \begin{split}
        \delta_{\ast}^{p-1}(w) &= [\delta_{\ast}^{p-1}(c_1), \delta_{\ast}^{p-1}(c_2)] \\
        &= \big[[\dots x_{i,p-1} + x_{i, p}\dots x_{k, p-1}+x_{k, p}\dots], [\dots x_{k,p-1}+x_{k,p}\dots x_{j,p-1}+x_{j, p} \dots]\big].
      \end{split}
    \end{equation*}

    Recall that by assumption $c_1$ and $c_2$ are iterated Whitehead products of the elements~$x_{i,p-1}$ with $i \in I$ and $x_{j,p-1}$ with $j \in J$ where $I, J \subseteq \{1, \dots, p-2\}$ and $I \cap J = \{k\}$ and $I \cup J = \{1, \dots, p-2\}$. 
    Furthermore, each $x_{i,p-1}$ with $i \in I$ appears exactly once in $c_1$ and similarly each $x_{j,p-1}$ with $j \in J$ appears exactly once in $c_2$.

    Again recall that $\quot{E_{p-1,p}^1}{\mathrm{tors}}$ is a subgroup of $\pi_{p}(\Embcos_{p-1}) \cong  \pi_p(\Conf_{p-1}(\RR^2 \times \udisc^1))$. 
    For an element $w \in \pi_p(\Conf_{p-1}(\RR^2 \times \udisc^1))$ of the form $\left[\dots[x_{i, p-1}+x_{i, p}, x_{j, p-1}+x_{j, p}]\dots \right]$ with $i \neq j$, we have by Proposition~\ref{prop:confrelation} the following equality
    \[
      \big[\dots[x_{i, p-1}+x_{i, p}, x_{j, p-1}+x_{j, p}]\dots\big] = \big[\dots[x_{i, p-1}, x_{j, p-1}]\dots\big] +\big[\dots[x_{i,p}, x_{j,p}]\dots\big].
    \]

    Thus by induction on the number of brackets the brackets, we have
    \begin{align*}
      \delta_{\ast}^{p-1}(c_1) &= [\dots x_{i, p-1}\dots x_{k, p-1}\dots] + [\dots x_{i, p}\dots x_{k,p} \dots] \\
      \delta_{\ast}^{p-1}(c_2) &= [\dots x_{k, p-1}\dots x_{j,p-1}\dots] + [\dots x_{k, p}\dots x_{j,p} \dots].
    \end{align*}
    and thus
    \begin{align*}
      \delta_{\ast}^{p-1}(w) &= \big[[\dots x_{i, p-1}\dots x_{k, p-1}\dots], [\dots x_{k, p-1}\dots x_{j, p-1}\dots]\big] \\
                    &+ \big[[\dots x_{i, p}\dots x_{k, p} \dots], [\dots x_{k, p}\dots x_{j, p} \dots]\big] \\
                    &+\big[[\dots x_{i, p-1}\dots x_{k, p-1}\dots], [\dots x_{k, p}\dots x_{j, p}\dots]\big] \\
                    &+ \big[[\dots x_{i, p}\dots x_{k, p} \dots], [\dots x_{k, p-1}\dots x_{j, p-1}\dots]\big].
    \end{align*}

    Let us now look at $\delta_{\ast}^{p-2}(w)$ and $\delta_{\ast}^{p}(w)$.
    We have
    \[
      \delta_{\ast}^{p}(w) = \big[[\dots x_{ik}\dots x_{k, p-1}\dots], [\dots x_{k, p-1}\dots x_{j, p-1}\dots]\big].
    \]

    Assume without loss of generality that $x_{p-2,p-1}$ appears in $c_1$, and write $c_1$ of the form~$c_1 = [\dots x_{i, p-1}\dots x_{k, p-1}\dots x_{p-2,p-1}\dots]$.
    We get 

    \begin{align*}
      \delta_{\ast}^{p-2}(w) &=\big[[\dots x_{i, p}\dots x_{k, p}\dots x_{p-2, p}+x_{p-1,p}\dots], [\dots x_{k, p}\dots x_{j, p}\dots]\big]\\ 
                    &=\big[[\dots x_{i, p}\dots x_{k, p}\dots x_{p-2, p}\dots], [\dots x_{k, p}\dots x_{j, p} \dots]\big] \\
                    &+\big[[\dots x_{i, p}\dots x_{k, p}\dots x_{p-1, p}\dots], [\dots x_{k, p}\dots x_{j, p}\dots]\big].
    \end{align*}

    Thus after cancelling with the terms of $(-1)^{p-2}\delta_{\ast}^{p-2}(w)$ and $(-1)^{p}\delta_{\ast}^{p}(w)$, the remaining term of $(-1)^{p-1}\delta_{\ast}^{p-1}(w)$ is
    \begin{multline*}
      (-1)^{p-1}\big[[\dots x_{i, p-1}\dots x_{k, p-1}\dots], [\dots x_{k, p}\dots x_{j, p} \dots]\big] \\+ (-1)^{p-1} \big[[\dots x_{i, p}\dots x_{k, p} \dots], [\dots x_{k, p-1}\dots x_{j, p-1}\dots]\big],
    \end{multline*}
    which is exactly $\partial^{p-1}(w)$ as desired.
  
  Finally one can prove Claim 1 for $k = p-2$ analogously.
  Then we can explicitly write down $\delta^{r}_{\ast}(w)$ for $r = p-3, p-2, p-1$ and $p$ and obtain the desired formula in the proposition.
\end{proof}

\begin{remark}\label{rmk:samebrackets}
  Note that in the calculations of the proof we only changed the indices of~$x_{i,p-1}$ for $i = 1, \dots, p-2$, whereas the bracketing of $c_1$ and $c_2$ was not changed at all.
  More explicitly, the expressions $\left[ [\dots x_{i', p}\dots x_{k,p}\dots], [\dots x_{k+1,p}\dots x_{j', p}\dots] \right]$ in the formula of $\partial^{k}(w)$ and $\left[[\dots x_{i,p-1}\dots x_{k,p-1}\dots], [\dots x_{k,p}\dots x_{j,p}\dots] \right]$
  in the formula of $\partial^{p-1}(w)$ have the same bracketing as the one of $c_1$.
  The term $\left[[\dots x_{i', p}\dots x_{k+1,p}\dots],[\dots x_{k,p}\dots x_{j', p}\dots]\right]$ in $\partial^k(w)$ and $\left[[\dots x_{i, p}\dots x_{k,p}\dots], [\dots x_{k,p-1}\dots x_{j,p-1}\dots]\right]$ in $\partial^{p-1}(w)$ have the same bracketing as $c_2$. 
  We note in advance here that the bracketing determines the shape of the unitrivalent graphs in the combinatorial interpretation in the next section.
\end{remark}


\section{Combinatorial interpretation}\label{sec:combidi}
\counterwithin{equation}{section}

Since we computed the abelian groups $E_{p, p}^1$, $E_{p-1, p}^1$ and the differential $d^1 \colon E_{p, p}^1 \to E_{p-1, p}^1$ of the spectral sequence associated to the cosimplicial model $\Embcos^{\bullet}$ in the previous section, we would also know $E_{p, p}^2$.
In this section we give a combinatorial interpretation of~$E_{p, p}^1$ and $E_{p-1, p}^1$ and the differentials $d^1$, based on the calculation of Proposition~\ref{prop:e1page} and Theorem~\ref{prop:d1simplified}. 
As a corollary, we obtain a graphic interpretation of the groups $E_{p, p}^2$.
From these interpretations we will see that this spectral sequence relates closely to the theory of Vassiliev invariants.

\begin{definition}\label{def:totalordutg}
  A \emph{unitrivalent graph} $\Gamma$ is a graph whose nodes have only degree $1$~or~$3$, together with a cyclic order on the edges at each node.
  We call the nodes of degree $1$~\emph{leaves} and nodes of degree $3$ \emph{trivalent nodes}.
  When $\Gamma$ has $n$ leaves, we define a \emph{labelling} (or total ordering) on $\Gamma$ to be a bijection of the set $\{1, 2, \dots, n\}$ to the set of leaves.
  Denote by $\mathrm{UTG}$ the collection of labelled unitrivalent graphs.
  We define the \emph{degree} of $\Gamma$ as the number of nodes divided by $2$.
\end{definition}

When we draw a labelled unitrivalent graph, we place the leaves on an oriented line, ordered according to the labelling.
Unless explicitly mentioned, the cyclic orders of the trivalent nodes are given counterclockwise.
See Figure~\ref{fig:unitrigraph} for an example of labelled unitrivalent graph.

\begin{figure}[!ht]
  \includegraphics[width=\textwidth]{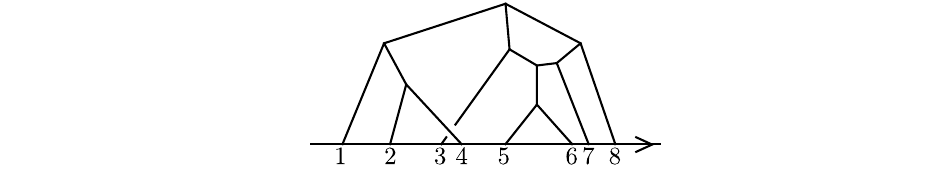}
  \caption{A unitrivalent graph of degree 8. The arrow is not part of the graph.}
  \label{fig:unitrigraph}
\end{figure}

\begin{definition}\label{def:asihx}
  We define the following relations on $\ZZ[\mathrm{UTG}]$:
  \begin{enumerate}
    \item Two labelled unitrivalent graphs $\Gamma_1$ and $-\Gamma_2$ are \emph{AS-related} if $\Gamma_1$ and $\Gamma_2$ are the same up to the cyclic order at one node.
      This is depicted in Figure~\ref{fig:asequi}.
    \item Let $\Gamma$ be a labelled unitrivalent graph.
      Let $e$ be an edge in $\Gamma$ between two trivalent nodes $v$ and $w$.
      Then $\Gamma$ is \emph{IHX-related} to the difference $\Gamma'-\Gamma''$ of the following two labelled unitrivalent graphs $\Gamma'$ and $\Gamma''$:
      Let $\{e, e_{v}' e_{v}''\}$ be the ordered set of edges at the node $v$, \ie we have $e < e_v' < e_v''$ (cyclic order).
      In the same way, let $\{e, e_{w}', e_{w}''\}$ be the ordered set of edges at the node $w$.
      The graph $\Gamma'$ arises from $\Gamma$ by deleting the edge $e$ and the nodes $v$ and $w$ of $\Gamma$, and adding an edge $e'$ and two trivalent nodes $v'$ and $w'$ such that the ordered set of edges at $v'$ are $\{e', e_v'',e_w'\}$ and the ordered set of edges at $w'$ are $\{e', e_w'', e_v'\}$.
      The unitrivalent graph $\Gamma''$ is constructed in a similar fashion.
      For $\Gamma''$ the ordered set of edges at the nodes $v'$ is $\{e' < e_v'' < e_w''\}$ and at the nodes $w'$ is $\{e' < e_w' < e_v'\}$.
      This is depicted in Figure~\ref{fig:ihxequi}.
  \end{enumerate}
\end{definition}

\begin{figure}[H]
\includegraphics[width=0.9\textwidth]{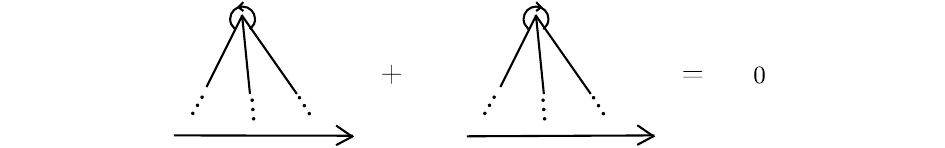}
\caption{A visualisation of the AS-relation.}
\label{fig:asequi}
\end{figure}

\begin{figure}[H]
\includegraphics[width=\textwidth]{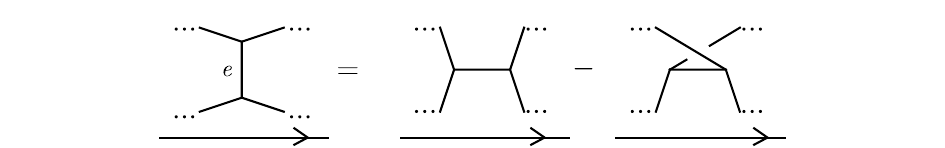}
\caption{A visualisation of the IHX-relation.}
\label{fig:ihxequi}
\end{figure}

\begin{definition}\label{def:Tcap}
  Denote by $\mathrm{UTT}_p$ the set of labelled unitrivalent trees of degree $p$.
  We denote by~$\Tca_p$ the abelian group generated by $\mathrm{UTT}_p$ modulo AS- and IHX-relations, \ie
  \[
    \Tca_p \coloneqq \quot{\ZZ \left[ \mathrm{UTT}_{p} \right]}{\sim_{\textnormal{AS}}, \sim_{\textnormal{IHX}}}.
  \]
\end{definition}

\begin{construction}\label{constr:eppcomb}
  For $p \geq 2$, denote by $T_{p}$ the set of iterated Whitehead products of the elements $x_{ip}$ with $ 1 \leq i < p$ such that $x_{ip}$ appears at most once in a given iterated Whitehead product.
  We are going to construct a one-to-one correspondence
  \[
    T_p \underset{\Phi_T}{\overset{\Psi_T}{\rightleftarrows}} \left\{\begin{aligned}
    &\text{labelled unitrivalent tree of degree at most $p-1$} \\ 
    &\text{together with a monotone bijection of their labelling}\\
    &\text{with a subset of $\{1, \dots, p\}$ containing $p$}
    \end{aligned}\right\}\!.
  \]
  
  Define the \emph{length} of $\tau \in T_p$ to be the total number of occurrences of $x_{i,p}$ with $i = 1, \dots, p-1$ in $\tau$.
  We will define $\Psi_{T}$ inductively on the length $n$ of $\tau$.
  Define $\Psi_T(x_{ip})$ to be the degree $1$ labelled unitrivalent tree consisting of two nodes labelled by $i$ and $p$ and an edge connecting them.
  Assume that for all $\tau_{k} \in T_p$ of length $k$ with $k \leq n-1 < p$, we have that $\Psi_T(\tau_k)$ is a degree $k-1$ unitrivalent tree with labelling $L_{\tau_k} \coloneqq \{i \in \NN \mid x_{ip} \textnormal{ appears in } \tau_k\} \cup \{p\}$. 
  For a tree $\tau_n = [\tau', \tau''] \in T_p$ of length $n$, we know that $\tau'$ and $\tau''$ are elements of $T_p$ and of length smaller than $n$. 
  By induction hypothesis, both $\Gamma' \coloneqq \Psi_T(\tau')$ and $\Gamma'' \coloneqq \Psi_T(\tau'')$ have a leaf with label $p$.
  We define the labelled unitrivalent tree $\Psi_T(\tau_n)$ to be the tree that arises by joining the tree $\Gamma'$ and $\Gamma''$ at the respective leaves labelled by $p$, and joining to this joint point a new leaf labelled by $p$.
  The set of labels of $\Psi_T(\tau_n)$ is $L_{\tau'} \cup L_{\tau''}$.
  This construction is depicted in Figure~\ref{fig:Liet1t2}, where also the cyclic order at the joint node is indicated.
  
  \begin{figure}[!h]
    \includegraphics[width=\textwidth]{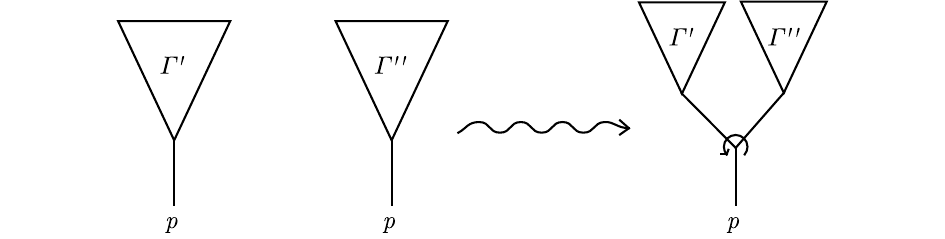}
    \caption{Unitrivalent trees $\Gamma'$, $\Gamma''$, and the unitrivalent tree obtained from $\Gamma'$ and $\Gamma''$ by joining their leaves labelled by $p$.}
    \label{fig:Liet1t2}
  \end{figure}
    
  Now we proceed to define the inverse map $\Phi_T$.
  For a labelled unitrivalent tree $\Gamma_1$ of degree $1$ with the set of labellings $\{i,p\}$, set $\Phi_T(\Gamma_1) = x_{ip}$.  
  Assume that we have already defined $\Phi_T$ for unitrivalent trees of degree smaller than $n$ with $n < p-1$.
  Every unitrivalent tree $\Gamma_n$ of degree $n$ can be depicted as in Figure~\ref{fig:Liet1t2}.
  Then define $\Phi_T(\Gamma_n) \coloneqq [\Phi_T(\Gamma'), \Phi_T(\Gamma'')]$, where $\Gamma'$ and $\Gamma''$ are the trees depicted in Figure~\ref{fig:Liet1t2}. 
  By construction, the two maps $\Psi_T$ and $\Phi_T$ are inverse to each other.
\end{construction}

\begin{proposition}\label{prop:eppcomb}
  Recall the group $E_{p,p}^1 \cong \bigoplus_{H}\ZZ$ from Proposition~\ref{prop:e1page}. 
  For $p \geq 2$, the construction above induces an isomorphism $E_{p,p}^1 \cong \Tca_{p-1}$ of groups.
\end{proposition}

\begin{lemma}\label{lem:eppalter}
  Let $J_p \subseteq T_p$ be the set of iterated Whitehead products of the elements $x_{ip}$ with $i = 1, \dots, p-1$ such that each $x_{ip}$ appears exactly once in an iterated Whitehead product.
  Then we have
  \[
    E_{p,p}^1 \cong \quot{\ZZ[J_p]}{\sim},
  \]
  where $\sim$ denotes the antisymmetry and Jacobi identity relations from Proposition~\ref{prop:whiteheadprod}.
\end{lemma}

\begin{proof}
  Recall $H$ from Proposition~\ref{prop:e1page}, and note that $H \subseteq J_p \subseteq E_{p,p}^1$. 
  Thus any element of $ J_p \setminus H$ can be written as linear combination of elements of $H$. 
  Furthermore, this linear combination is produced by applying the Jacobi identity and antisymmetry relation to the element, \cf \cite[Theorem~3.1]{Hal50}. 
  As a result, we obtain the desired a group isomorphism~$E_{p,p}^1 \cong \quot{\ZZ[J_p]}{\sim}$.
\end{proof}

\begin{proof}[Proof of Proposition~\ref{prop:eppcomb}]  
  We are going to define group homomorphisms 
  \[
    \ZZ[J_p] \underset{\phi_T}{\overset{\psi_T}{\rightleftarrows}}\ZZ[\textnormal{UTT}_{p-1}],
  \]
  such that the induced map $\overline{\psi}_T$ and $\overline{\phi}_T$ on the quotients $E_{p,p}^1$ and $\Tca_{p-1}$ are inverse to each other.
  We~will define our morphisms on generators and extend linearly to the whole group.
  
  For $p = 2$, define $\psi_T(x_{12}) = \Gamma$ and $\phi_T(\Gamma) = x_{12}$, where $\Gamma$ is the labelled unitrivalent tree of degree $1$ and with labelling set $\{1,2\}$. 
  There is no relation to consider when passing to the quotients $E_{2,2}^{1}$ and $\Tca_{1}$.
  Thus $\bar{\psi}_T$ and $\bar{\phi}_T$ are inverse to each other by definition.
  
  For $p \geq 3$ and $v = [v_1, v_2] \in J_p$, define 
  \[
    \psi_T([v_1,v_2]) = (-1)^{\#{L_{1}} + \# \left( L_1 \times_{>} L_2 \right)}\Psi_T([v_1,v_2])
  \]
  where\footnote{Note that $L_i$ is the set of labels of the tree $\Psi_T(v_i)$, and thus $\# L_i$ is the number of leaves of this tree.} $L_i \coloneqq \{j \in \NN \mid x_{j p} \textnormal{ appears in } v_i\} \cup \{p\}$ and 
  \[ 
  L_1 \times_{>} L_2 \coloneqq \{(a, b) \in L_{1} \times L_{2} \mid a > b,\ \text{and}\ a,b \not\in L_1 \cap L_2 \}.
  \]
    
  To see that the anti-symmetry of the Whitehead product corresponds to the AS-relation, we look at $\psi_T([v_1,v_2] + (-1)^{ \# ( L_1 \times L_2) }[v_2, v_1])$ which equals 
  \begin{equation}\label{eq:antisymm}
    (-1)^{\#{L_{1}}+\# \left( L_1 \times_{>} L_2 \right)} \Psi_T([v_1, v_2]) + (-1)^{\# \left(L_1 \times L_2\right) + \# \left( L_2 \times_{>} L_1 \right) + \# {L_{2}}} \Psi_T([v_2,v_1]). 
  \end{equation}
  Recall from Construction~\ref{constr:eppcomb} that the only difference between the tree $\Psi_{T}([v_1,v_2])$ and the tree $\Psi_{T}([v_2, v_1])$ is the cyclic order at the trivalent node which is adjacent to the leaf with label~$p$, \ie $\Psi_{T}([v_1,v_2]) \sim_{\textnormal{AS}} -\Psi_{T}([v_2, v_1])$.
  Note that the sum of signs in Equation~$\ref{eq:antisymm}$ is 
  \begin{multline*}
    \#{L_1} + \# \left(L_1 \times_{>} L_2 \right)+ \# \left( L_1 \times L_2 \right) + \# \left( L_2 \times_{>} L_1 \right) + \#{L_2} \\=
    \#{L_1} + \# \left( L_1 \times L_2 \right) + \#{L_2} + (\#{L_1}-1)(\#{L_2}-1),
  \end{multline*}
  because $\# (L_1 \times_{>} L_2) + \# \left( L_2 \times_{>} L_1 \right) = (\#{L_1}-1)(\#{L_2}-1)$. 
  Thus we obtain that the element in Formula~\ref{eq:antisymm} AS-related to 0.
  
  As for the Jacobi identity, take $v = [v_1, [v_2, v_3]] \in J_p$ with $v_i \in \pi_{l_i}(\Conf_p(\RR^2 \times \udisc^1))$ where $l_i = \#{L_i} $. 
  Then 
  \begin{multline}\label{eq:jacobi}
    \psi_T\left((-1)^{(l_3-1)l_1}[v_1,[v_2,v_3]] + (-1)^{(l_1-1)l_2}[v_2, [v_3, v_1]] + (-1)^{(l_2-1)l_3}[v_3, [v_1,v_2]]\right) \\
    =  (-1)^{\epsilon_1}\Psi_T([v_1,[v_2,v_3]]) + (-1)^{\epsilon_2}\Psi_T([v_2, [v_3, v_1]]) + (-1)^{\epsilon_3}\Psi_{T}([v_3, [v_1,v_2]]),
  \end{multline}
  where $\epsilon_1$, $\epsilon_2$ and $\epsilon_3$ are the suitable signs.
  Again by Construction~\ref{constr:eppcomb}, we have that 
  \[
    (-1)^{\epsilon_1}[\Psi_T([v_1,[v_2,v_3]])] + (-1)^{\epsilon_2}[\Psi_T([v_2, [v_3, v_1]])] + (-1)^{\epsilon_3}[\Psi_{T}([v_3, [v_1,v_2]])] 
  \]
  is IHX-related to $0$, and thus the element in Equation~\ref{eq:jacobi} is IHX-related to $0$.
  Therefore,~$\bar{\psi}_T$ is well-defined.
  
  Let $\Gamma_{p-1}$ be a labelled unitrivalent tree of degree $p-1$, drawn as in Figure~\ref{fig:Liet1t2}, define
  \[
    \phi_T(\Gamma_{p-1}) =(-1)^{\#{L_1} + \# \left( L_1 \times_{>} L_2 \right)} \Phi_T(\Gamma_{p-1}),
  \]
  where $L_1$ and $L_2$ is the set of labels of $\Gamma'$ and $\Gamma''$ respectively.
  
  Similar to the discussion of $\overline{\psi}_T$, one can show that $\overline{\phi}_T$ is well-defined. 
  Therefore, the maps $\overline{\psi}_T$ and $\overline{\phi}_T$ are inverses to each other by construction.
\end{proof}

\begin{remark}
  Based on the methods in \cite[Section 4]{Con08}, Construction~\ref{constr:eppcomb} and Proposition~\ref{prop:eppcomb} are generalisations of \cite[Definition 4.5, Proposition 4.6]{Con08} and \cite[Proposition 4.7]{Con08} respectively. 
  Our proofs also add missing details of \cite[Proposition~4.7]{Con08}.
\end{remark}

\begin{definition}\label{def:ijmarkedFdiag}
  Let $i,j \in \NN$ with $i \neq j$.
  An \emph{$(i,j)$-marked unitrivalent graph} $\Gamma_{ij}$ is a unitrivalent graph of degree $j$ together with two marked nodes $v_i$ and $v_j$ that satisfy the following properties:
  \begin{enumerate}
    \item The underlying graph of $\Gamma$ is connected and has exactly one simple cycle\footnote{A simple cycle is defined as a loop in the graph with no repetitions of nodes and edges allowed, other than the repetition of the starting and ending nodes.}.
    \item The two marked nodes $v_i$ and $v_j$ are adjacent to the leaf with label $i$ and $j$, respectively, and lie on the simple cycle.
  \end{enumerate}
  Denote by $\textnormal{UTG}_{i,j}$ the set of $(i,j)$-marked unitrivalent graphs.
\end{definition}

In Figure \ref{fig:unitrigraph}, we can mark the nodes adjacent to the leaf with lable 3 and 8, respectively, and obtain a (3, 8)-marked unitrivalent graph, see Figure \ref{fig:combi-feynmanloop}.

\begin{figure}[!ht]
    \includegraphics[width=\textwidth]{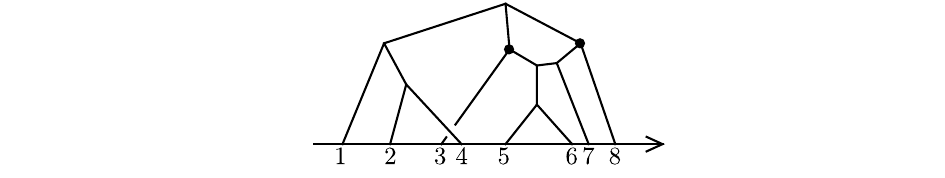}
    \caption{A $(3,8)$-marked unitrivalent graph. The blacks dots indicate the marked nodes. The arrow is not part of the graph.}
    \label{fig:combi-feynmanloop}
\end{figure}

\begin{definition}
  For $p \geq 3$, define $\Dca_p$ to be the abelian group generated by the collection of $(i,p)$-marked unitrivalent graphs with $1 \leq i < p$, modulo AS- and IHX\textsuperscript{sep}-relations, \ie 
  \[
    \Dca_p \coloneqq \quot{\ZZ[\cup_{i =1}^{p-1} \mathrm{UTG}_{i,p}]}{\sim_{\textnormal{AS}}, \sim_{\textnormal{IHX}^{\mathrm{sep}}}},
  \]
  where the IHX\textsuperscript{sep}-relation is the usual IHX-relation, except that the edge $e$, which appears in Definition~\ref{def:asihx}.\rom{2}, is not allowed to be an edge adjacent to the marked nodes. 
\end{definition}

\begin{construction}\label{constr:ep-1pcomb}
  Recall from Proposition~\ref{prop:separateindices} the definition of $D^{\mathrm{sep}}_{p}$ for $p \geq 4$. 
  We are going to construct a one-to-one correspondence
  \[
    D^{\mathrm{sep}}_{p} \underset{\Phi_{D}}{\overset{\Psi_D}{\rightleftarrows}} \bigcup_{i = 1}^{p-2} \mathrm{UTG}_{k,p-1}.
  \]
  We use $\Psi_{T}$ and $\Phi_{T}$ from Construction~\ref{constr:eppcomb} to define $\Psi_{D}$ and $\Phi_{D}$.
  
  For $w=[c_1, c_2] \in D^{\mathrm{sep}}_{p}$, say with repeated occurrence of $x_{k,p-1}$, apply $\Psi_T$ to $c_1$ and to $c_2$.
  We obtain the following two labelled unitrivalent trees $\Gamma_1 \coloneqq \Psi_{T}(c_1)$ and $\Gamma_2 \coloneqq  \Psi_T({c_2})$.
  \[
    \includegraphics[width=\textwidth]{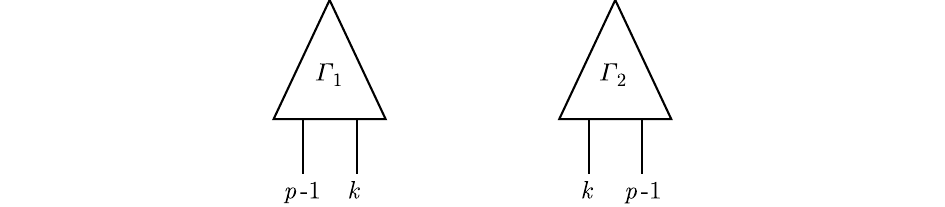}
  \]
  
  Define $\Psi_D(w)$ to be the following $(k,p-1)$-marked unitrivalent graph.
  \begin{figure}
    \centering
    \includegraphics[width=\textwidth]{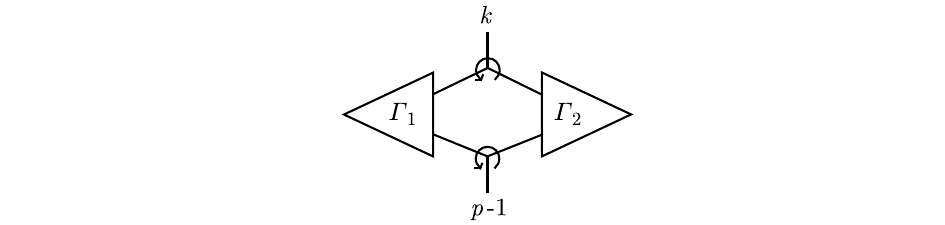}
    \caption{A visualisation of the labelled unitrivalent graph $\Psi_D(w)$.}
    \label{fig:Liec1c2}
  \end{figure}

  For a $(k,p-1)$-marked unitrivalent diagram $\Gamma_{k,p-1}$, we can draw $\Gamma_{k,p-1}$ as in Figure~\ref{fig:Liec1c2}. 
  Then define $\Phi_D(\Gamma_{k,p-1}) \coloneqq [\Phi_T(\Gamma_{1}), \Phi_T(\Gamma_{2})]$.
  By construction the maps $\Psi_D$ and $\Phi_D$ are inverse to each other, since $\Phi_{T}$ and $\Psi_{T}$ are each other's inverse, \cf~Construction~\ref{constr:eppcomb}.
\end{construction}

\begin{proposition}
  Let $w = \left[ [c_1, [c_2, c_3]], c_4 \right]$ be an iterated Whitehead product in $D^{\mathrm{sep}}_p$, say with $c_i \in \pi_{l_i}(\Conf_{p-1}\left(\RR^2 \times \udisc^1\right))$. 
  We have in $\ZZ[D^{\mathrm{sep}}_{p}]$ the \emph{separated Jacobi identity}, \ie
  \[
    (-1)^{(l_3-1)l_1}\big[ [ c_1,[c_2,c_3] ], c_4 \big] + (-1)^{(l_1-1)l_2}\big[ [ c_2,[c_3,c_1] ], c_4 \big] + (-1)^{(l_2-1)l_3}\big[ [ c_3,[c_1,c_2] ], c_4 \big] = 0.
  \]  
\end{proposition}

\begin{proof}
  Apply the Jacobi identity to $\left[c_1, [c_2, c_3] \right]$.
\end{proof}

\begin{remark}
  Note that the usual Jacobi identity is not well-defined in $\ZZ[D^{\mathrm{sep}}_{p}]$.
\end{remark}

\begin{proposition}\label{prop:ep-1pcomb}
  Recall the group $\quot{E_{p-1,p}^1}{\mathrm{tors}}$ from Proposition~\ref{prop:separateindices}.
  For $p \geq 4$, the construction above induces an isomorphism $\quot{E_{p-1,p}^1}{\mathrm{tors}} \cong \Dca_{p-1}$ of groups.
\end{proposition}

Similar to Lemma~\ref{lem:eppalter}, we obtain the following presentation of $\quot{E_{p-1,p}^1}{\mathrm{tors}}$.

\begin{lemma}\label{lem:ep-1pcomb}
  We have 
  \[
    \quot{E_{p-1,p}^1}{\mathrm{tors}} \cong \quot{{\ZZ[D^{\mathrm{sep}}_{p}]}}{\sim},  
  \]
  where $\sim$ denotes relation induced by antisymmetry and the separated Jacobi Identity.  \qed
\end{lemma}

\begin{proof}[Proof of Proposition~\ref{prop:ep-1pcomb}]
  Similar as in the proof of Proposition~\ref{prop:eppcomb}, we can define group homomorphisms
  \[
    \ZZ[D^{\mathrm{sep}}_p] \underset{\phi_D}{\overset{\psi_D}{\rightleftarrows}}\ZZ[\bigcup_{i = 1}^{p-1}\mathrm{UTG}_{i,p-1}].
  \]
  such that the induced map $\overline{\psi}_D$ and $\overline{\phi}_D$ are on the quotients are inverses to each other.
  Let $w = [c_1, c_2] \in D^{\mathrm{sep}}_{p}$, say with repeated occurrence of $x_{k,p-1}$. 
  For $i = 1, 2$, define the set of labels $L_i \coloneqq \{j \in \NN \mid x_{j,p-1} \textnormal{ appears in } c_i\} \cup \{p-1\}$.
  Define $\psi_D(w) \coloneqq (-1)^{ \# (L_1 \times_{>} L_2)} \Psi_D(w)$.
  
  Let $\Gamma_{k,p-1} \in \mathrm{UTG}_{k,p-1}$, for example drawn as in Figure~\ref{fig:Liec1c2}. 
  Denote by $L_i$ the labelling of $\Gamma_{i}$ with $i = 1, 2$. 
  Define $\phi_D(\Gamma_{k,p-1}) \coloneqq (-1)^{\# \left( L_1 \times_{>} L_2 \right)} \Phi_D(\Gamma_{k,p-1})$. 
  
  Similar as in the proof of Proposition~\ref{prop:eppcomb} we can check by explicit computation that the induced maps $\bar{\psi}_{D}$ and $\bar{\phi}_{D}$ on the quotients $\quot{E_{p-1,p}^1}{\mathrm{tors}}$ and $\Dca_{p-1}$ are well-defined.
  More precisely, antisymmetry corresponds to AS-relation and the separated Jacobi identity corresponds to IHX\textsuperscript{sep}-relation.
  By construction the maps $\overline{\psi}_D$ and $\overline{\phi}_D$ are inverse to each~other. 
\end{proof}

\begin{definition}\label{def:stu}
  Define the following relations on $\ZZ[\textnormal{UTG}]$:
  \begin{enumerate}
    \item Let $\Gamma$ be a labelled unitrivalent graph. 
      Let $e$ be the edge connecting the leaf labelled by $n$ and the adjacent trivalent node $v$.
      Then $\Gamma$ is \emph{$STU$-related} to the difference $\Gamma'-\Gamma''$ of the following two labelled unitrivalent graphs $\Gamma'$ and $\Gamma''$
      Denote by $\{e, e_1, e_2\}$ the ordered set of edges at the node $v$, \ie $e < e_1 < e_2$ (cyclic order).
      The graph $\Gamma'$ arise from $\Gamma$ the following steps:
      First, delete the edge $e$, the node $v$ and the leaf labelled by $n$. 
      Second, add two leaves, labelled by $n$ and $n + 1$ with adjacent edges $e_2$ and $e_1$ respectively. 
      Third, relabel the leaves labelled with $m$ by $m + 1$ if $m > n$.
      The graph $\Gamma''$ is constructed similar to $\Gamma'$.
      For $\Gamma''$ the adjacent edges of the leaves labelled by $n$ and $n+1$ are $e_1$ and $e_2$ respectively.
      The STU-relation is depicted in Figure~\ref{fig:stuequi}.
      \begin{figure}[!ht]
        \includegraphics[width=\textwidth]{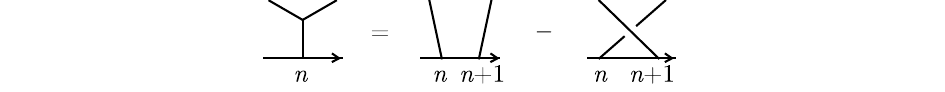}
        \caption{A visualisation of the STU-relation.}
        \label{fig:stuequi}
      \end{figure}
    \item Let $\Gamma$ be a labelled unitrivalent graph.
      Then $\Gamma_1'- \Gamma_1''$ is \emph{STU\textsuperscript{2}-related} to $\Gamma_2' - \Gamma_2''$, where $\Gamma_1'-\Gamma_1''$ is obtained by performing the STU-relation at the leaf of $\Gamma$ labelled by $n$, and $\Gamma_2'-\Gamma_2''$ is obtained by performing the STU-relation at the leaf of $\Gamma$ labelled by $m$.
      The STU\textsuperscript{2}-relation is depicted in Figure~\ref{fig:stu2equi}. 
      \begin{figure}[!ht]
        \includegraphics[width=\textwidth]{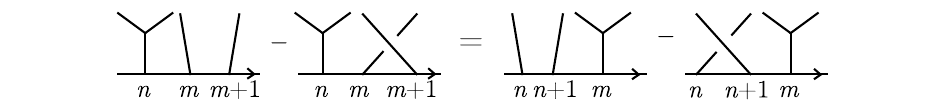}
        \caption{A visualisation of the STU\textsuperscript{2}-relation.}
        \label{fig:stu2equi}
      \end{figure}
  \end{enumerate}
\end{definition}

\begin{remark}
  Note that the STU\textsuperscript{2}-relation is finer than the STU-relation.
\end{remark}
\begin{proposition}
  The STU\textsuperscript{2}-relation is well-defined on $\ZZ[\mathrm{UTT}_{p}]$ for $p \geq 1$.
\end{proposition}

\begin{proof}
  This is because the STU\textsuperscript{2}-relation does not change the connectivity and degree of the unitrivalent graphs, and it also does not add simple cycles to the graphs.
\end{proof}

Recall the computation of the differential $d^1 \colon E_{p-1,p}^1 \to E_{p,p}^1$ from Proposition~\ref{prop:d1simplified}.
Note that in the computation only $d^1|_{\quot{E_{p-1,p}^1}{\mathrm{tors}}}$ is relevant, and by abuse of notation we will write $d^1$ instead of $d^1|_{\quot{E_{p-1,p}^1}{\mathrm{tors}}}$ in the following.
By Proposition~\ref{prop:eppcomb} and Proposition~\ref{prop:ep-1pcomb} we can consider $d^1$ as a map between two groups of unitrivalent graphs as follows
\[
  \Psi_{T} \circ d^1 \circ \Phi_{D} \colon \Dca_{p-1} \to \Tca_{p-1}.
\]
By abuse of notation we will also denote this map between $\Dca_{p-1}$ and $\Tca_{p-1}$ by $d^1$.
The following proposition describes what $d^1$ means on the level of unitrivalent graphs.

\begin{theorem}\label{thm:d1comb}
  Viewed on generators, the differential $d^1$ maps a $(k,p-1)$-marked unitrivalent graph~$\Gamma_{k, p-1}$ to the linear combination $\Gamma_{p}^{1} - \Gamma_{p}^{2} - (\Gamma_{k}^{2} - \Gamma_{k}^{1})$ of labelled unitrivalent trees, where 
  \begin{enumerate}
    \item $\Gamma_{p}^1 - \Gamma_p^2$ is obtained from performing the STU-relation on $\Gamma_{k, p-1}$ at the edge connecting the leaf labelled by $p-1$ and the marked node $v_{p-1}$, and
    \item $\Gamma_{k}^2 - \Gamma_k^1$ is obtained from performing the STU-relation on $\Gamma_{k, p-1}$ at the edge connecting the leaf labelled by $k$ and the marked node $v_{k}$.
  \end{enumerate} 
\end{theorem}

For a pictorial illustration of $d^1$, see Figure~\ref{fig:d1form}.
\begin{figure}
\includegraphics[width=\textwidth]{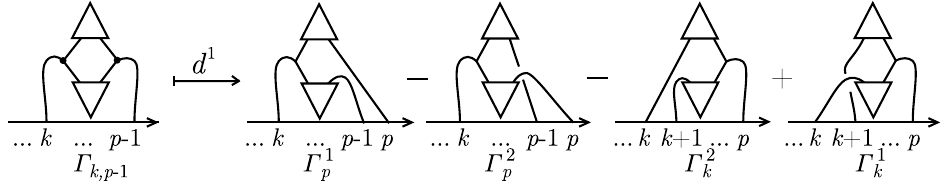}
\caption{An example of $d^1$ applied to a $(k,p-1)$-marked unitrivalent graph. The triangles are placeholders for subgraphs, which stay unmodified.}
\label{fig:d1form}
\end{figure}

\begin{proof}
  Denote by $w = [c_1, c_2] \in D^{\mathrm{sep}}_p$, say with repeated occurrence of $x_{k,p-1}$, a element in~$E_{p-1, p}^{1}$ that corresponds to $\Gamma_{k, p-1}$ under the isomorphism from Proposition~\ref{prop:ep-1pcomb}. 
  Write (without loss of generality) $c_1 = [\dots x_{i,p-1} \dots x_{k,p-1} \dots]$ and $c_2 = [\dots x_{k,p-1} \dots x_{j,p-1} \dots]$.
  Recall the notations from the proofs of Proposition~\ref{prop:eppcomb} and Proposition~\ref{prop:ep-1pcomb}, we have
  \begin{align*}
    \overline{\psi}_T(c_1) &= (-1)^{\epsilon_1}\Gamma_{1} \\
    \overline{\psi}_T(c_2) &= (-1)^{\epsilon_2}\Gamma_{2} \\
    \overline{\psi}_D(w) &= (-1)^{\epsilon_w} \Gamma_w,
  \end{align*}
  where $\Gamma_1$, $\Gamma_2$ and $\Gamma_w$ are depicted in Figure~\ref{fig:gamma12w}.
  We will discuss the signs $\epsilon_1$, $\epsilon_2$ and $\epsilon_w$ at the end.
  \begin{figure}[H]
    \includegraphics[width=\textwidth]{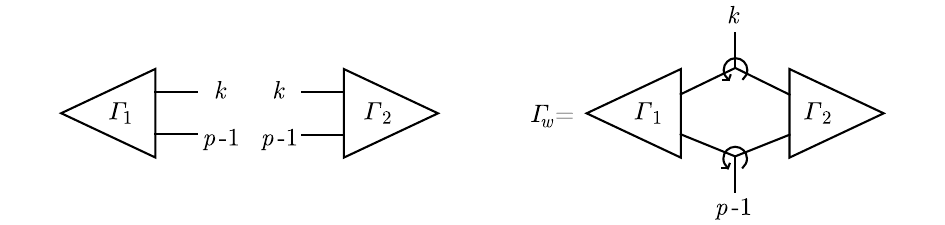} 
    \caption{Visualisations of $\Gamma_1$, $\Gamma_2$ and $\Gamma_w$.}
    \label{fig:gamma12w}
  \end{figure}    
 
  Recall the formula for $\partial^{k}(w)$ from Proposition~\ref{prop:d1simplified}. 
  We write
  \[
    \partial^{k}(w) = (-1)^{k}[c_1^{k}, c_2^{k+1}] + (-1)^{k}[c_{1}^{k+1}, c_2^k],
  \]
  where $c_1^{k}$, $c_{2}^{k+1}$, $c_1^{k+1}$ and $c_2^{k}$ correspond exactly to the four Whitehead brackets in the formula of $\partial^{k}(w)$.
  Thus we have
  \begin{align*}
    \overline{\psi}_T(c_1^k) &= (-1)^{\epsilon_1}\Gamma_{1}^k \\
    \overline{\psi}_T(c_2^{k+1}) &= (-1)^{\epsilon_2}\Gamma_{2}^{k+1} \\
    \overline{\psi}_T(c_1^{k+1}) &= (-1)^{\epsilon_1}\Gamma_{1}^{k+1} \\
    \overline{\psi}_T(c_2^{k}) &= (-1)^{\epsilon_2}\Gamma_{2}^{k},
  \end{align*}
  where $\Gamma_1^k$, $\Gamma_2^{k+1}$, $\Gamma_1^{k+1}$ and $\Gamma_2^{k}$ are the labelled unitrivalent trees depicted below
  \[
    \includegraphics[width=\textwidth]{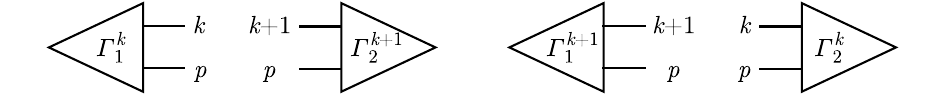}
  \]
       
  Note that the underlying unlabelled trees of $\Gamma_{i}^k$ and $\Gamma_i^{k+1}$ for $i = 1, 2$ are the same as the ones for $\Gamma_i$ respectively.
  Furthermore, we have 
  \begin{align*}
    \overline{\psi}_T([c_1^{k}, c_2^{k+1}]) &= (-1)^{\epsilon_k^1}\Gamma_k^1 \\
    \overline{\psi}_T([c_1^{k+1}, c_2^k]) &= (-1)^{\epsilon_k^2} \Gamma_k^2
  \end{align*}
  where $\Gamma_k^1$ and $\Gamma_k^2$ are the labelled unitrivalent trees depicted in Figure~\ref{fig:gammak}.
  \begin{figure}[!ht]
    \includegraphics[width=\textwidth]{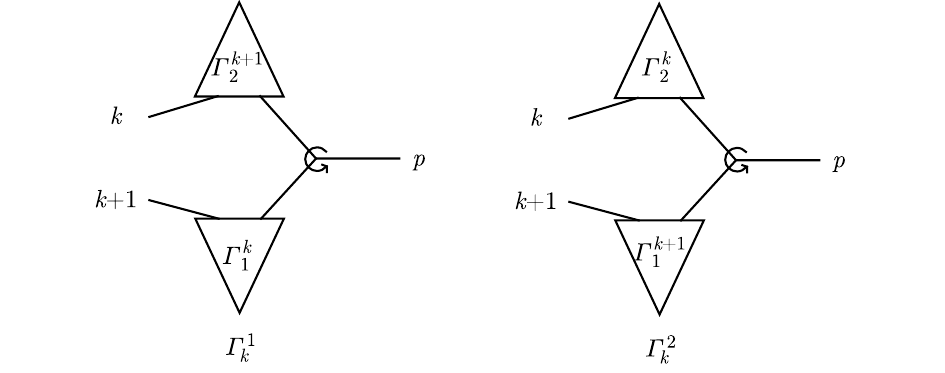}
    \caption{Descriptions of the labelled unitrivalent trees $\Gamma_k^1$ and $\Gamma_k^2$.}
    \label{fig:gammak}
  \end{figure}
  
  We perform the same steps for $\partial^{p-1}(w) = (-1)^{p-1}[c_1, c_2^{p}] + (-1)^{p-1}[c_{1}^{p}, c_2]$, thus obtain
  \begin{align*}
    \overline{\psi}_T(c_1^{p}) &= (-1)^{\epsilon_1}\Gamma_{1}^{p} \\
    \overline{\psi}_T(c_2^{p}) &= (-1)^{\epsilon_2}\Gamma_{2}^{p},
  \end{align*}
  where $\Gamma_1^{p}$ and $\Gamma_2^{p}$ are the labelled unitrivalent trees depicted in Figure~\ref{fig:gammap}. 
  \begin{figure}[!ht]
     \includegraphics[width=\textwidth]{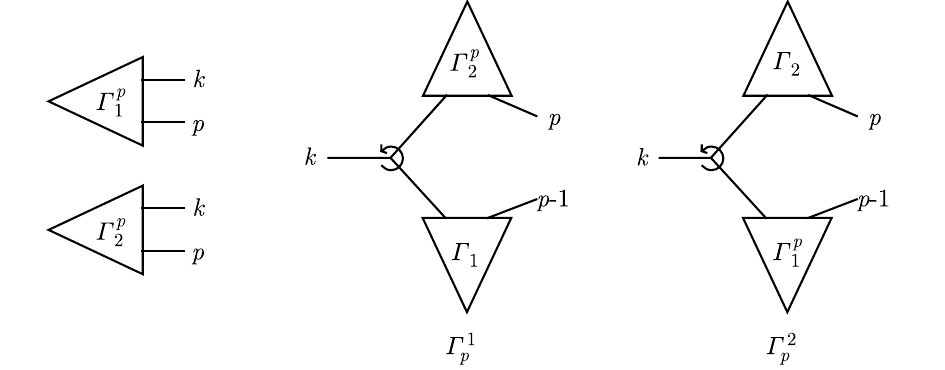}
     \caption{Descriptions of the labelled unitrivalent trees $\Gamma_1^p, \Gamma_2^p, $$\Gamma_p^1$ and $\Gamma_p^2$}
     \label{fig:gammap}
  \end{figure}
  
  Furthermore, $\overline{\psi}_T([c_1, c_2^{p}])= (-1)^{\epsilon_p^1}\Gamma_p^1$ and $ \overline{\psi}_T([c_1^{p}, c_2]) = (-1)^{\epsilon_p^2} \Gamma_p^2$, where $\Gamma_p^1$ and $\Gamma_p^2$ are depicted\footnote{Here we join the trees at the leaves labelled by $k$. Note that Construction~\ref{constr:eppcomb} only considers the case where we join at the leaves labelled by $p$. The construction for $k$ inplace $p$ is analogous, however, one has to check that both constructions apply to the same tree yield the same iterated Whitehead product.} in Figure~\ref{fig:gammap}.
  Reviewing the leaves of $\Gamma_k^i$ and $\Gamma_p^i$ as placed on the oriented line, we get the following illustrations of the trees
  \[
    \includegraphics[width=\textwidth]{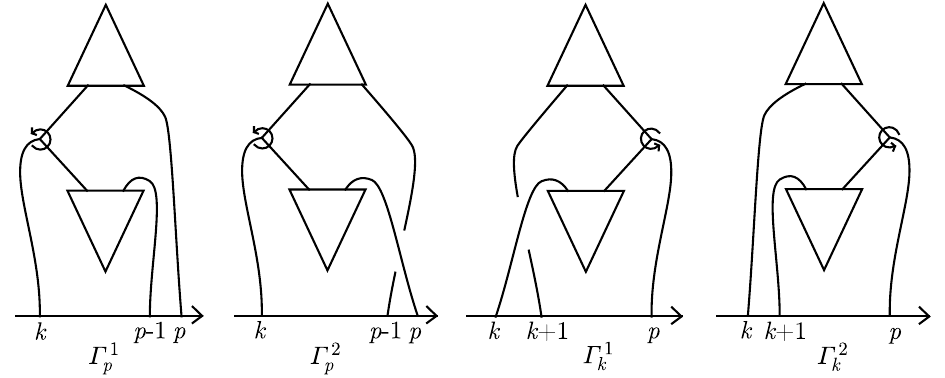}
  \]
       
  Thus, we have that
  \[
    \overline{\psi}_T(d^1(w)) = (-1)^{k} \big( (-1)^{\epsilon_k^1}\Gamma_{k}^1 + (-1)^{\epsilon_k^2}\Gamma_{k}^2 \big) +(-1)^{p-1} \big((-1)^{\epsilon_p^1}\Gamma_{p}^1 + (-1)^{\epsilon_p^2}\Gamma_{p}^2 \big).
  \] 
  Now let us check the signs.
  For $i = 1, 2$, let $L_{i}$ be the set of labels of $\Gamma_i$. 
  Thus we have 
  \[
    \epsilon_w = \epsilon_1 + \epsilon_2 + \# (L_1 \times_{>} L_2).
  \]
  where $\epsilon_1$ and $\epsilon_2$ appeared right at the beginning of the proof.
  Let $L_{i}^{j}$ be the set of labels of $\Gamma_{i}^{j}$ and $L_{i}^{p}$ be the set of labels of $\Gamma_{i}^{p}$, for $j \in \{k, k+1\}$ and $i \in \{1, 2\}$.
  Then we have 
  \[
    \epsilon_k^1 = \epsilon_1 + \epsilon_2 + \#{L_1^k} + \# (L_1^k \times_{>} L_2^{k+1}),
  \]
  \[
    \epsilon_k^2 = \epsilon_1 + \epsilon_2 + \#{L_1^{k+1}} + \# (L_1^{k+1} \times_{>} L_2^{k}), 
  \]
  \[
    \epsilon_p^1 = \epsilon_1 + \epsilon_2 + \#{L_1} + \# (L_1 \times_{>} L_2^p),
  \]
  \[
    \epsilon_p^2 = \epsilon_1 + \epsilon_2 + \#{L_1^p} + \# (L_1^p \times_{>} L_2).
  \]
  Recall Remark~\ref{rmk:samebrackets} and thus observe that $\#{L_1} = \#{L_1^k} = \#{L_1^{k+1}} = \#{L_1^p}$.
  Furthermore, we have $\# (L_1^{k+1} \times_{>} L_2^k) = \# (L_1^k \times_{>} L_2^{k+1}) + 1$, where the $+1$ arises from the pair $(k+1, k)$, and $\#( L_1^p \times_{>} L_2) = \# (L_1 \times_{>} L_2^p) + 1$, where the $+1$ arises from the pair $(p,p-1)$.
  Thus we can write $\overline{\psi}_T(d^1(w))$ as 
  \begin{equation}\label{eq:d1formel}
    \overline{\psi}_T(d^1(w)) = (-1)^{k+ \epsilon_k^2}(\Gamma_k^2 - \Gamma_k^1) + (-1)^{p-1+\epsilon_p^1}(\Gamma_p^1 - \Gamma_p^2).
  \end{equation}
  For the signs we use 
  \begin{align*}
    \# \left( L_1^{k+1} \times_{>} L_2^{k} \right) &= \# \left((L_1 \setminus \{k\}) \times_{>} (L_2 \setminus \{k\}) \right) \\
     &+ \# \left( \{k+1\} \times_{>} (L_2 \setminus \{k\}) \right) \\
     &+ \# \left( (L_1 \setminus \{k\}) \times_{>} \{k\} \right), \\
    \# \left( L_1 \times_{>} L_2^p \right) &= \# \left( ( L_1 \setminus \{k\} ) \times_{>} (L_2 \setminus \{k\}) \right) + \# (L_2 \setminus \{k,p-1\}), \\
    \# (L_2 \setminus \{k,p-1\}) &= \# (\{k+1\} \times_{>} \left(L_2 \setminus \{k\}\right)) + \# (\left(L_2 \setminus \{k\}\right) \times_{>} \{k+1\}) \\
    p-k-2 &= \# \left( ( L_2 \setminus \{k\} ) \times_{>} \{k+1\} \right) + \# \left( ( L_1 \setminus \{k\}) \times_{>} \{k\} \right).
  \end{align*}
  Thus we we have 
  \[
    k+\epsilon_k^2 + p-1 + \epsilon_p^1 = 2p-3,
  \]
  \ie the signs before $(\Gamma_k^2 - \Gamma_k^1)$ and $(\Gamma_p^1 - \Gamma_p^2)$ in Equation~\ref{eq:d1formel} are different.
  
  Consider the sign in front of $\Gamma_{k,p-1}$ and $\Gamma_p^1$ in the equations $\overline{\psi}_D([c_1, c_2]) = (-1)^{\epsilon_w} \Gamma_{k,p-1}$ and $\overline{\psi}_{T}([c_1, c_2^p]) = (-1)^{\epsilon_{p}^1} \Gamma_p^1$.
  First recall that
  \begin{align*}
    \epsilon_w &= \epsilon_1 + \epsilon_2 + \# \left( L_1 \times_{>} L_2 \right), \\
    \epsilon_p^1 &= \epsilon_1 + \epsilon_2 + \# L_1 + \# \left( L_1 \times_{>} L_2^p \right).
  \end{align*}
  Observe that we have $\# \left( L_1 \times_{>} L_2^p \right) = \# \left( L_1 \times_{>} L_2 \right) + \#\left( L_2 \setminus \{k,p\}\right)$, where the term $\# \left( L_2 \setminus \{k,p\}\right)$ arises because of the set of pairs $\{(p-1, x) \in L_1 \times L_2^p \mid x \neq k,p\}$. 
  Moreover we know that $\#{L_1} + \# \left( L_2 \setminus \{k, p\} \right) = p-1$.
  So the sign difference between $\overline{\psi}_D([c_1, c_2])$ and $\overline{\psi}_{T}([c_1, c_2^p])$ is $(-1)^{p-1}$, which cancels with the sign in front of $[c_1, c_2^p]$ in the expression of $\partial^{p-1}(w)$.
  Therefore, $\epsilon_w$ and $\epsilon_p^1$ have the same parity.
  In other words, $\Gamma_{k,p-1}$ and $\Gamma_p^1$ have the same sign. 
  
  In conclusion, we are justified to write
  \[
  d^1\left(\Gamma_{k, p-1}\right) = \Gamma_{p}^1 - \Gamma_{p}^2 - \left(\Gamma_k^{2} - \Gamma_{k}^1\right).  \qedhere
  \]
\end{proof}

\begin{corollary}[Conant]\label{cor:d1imgcomb}
  Let $\tau \in \Tca_{p-1}$, then $\tau \in \im(d^1)$ if and only if $\tau$ is STU\textsuperscript{2}-related to 0.
\end{corollary}

\begin{proof}
  ``$\Leftarrow$'' It follows from the formula for $d^1$ in the previous theorem that the image of the generators under $d^1$ are STU\textsuperscript{2}-related to 0.
  Thus, any element in $\im(d^1)$ is also STU\textsuperscript{2}-related to 0.
  
  ``$\Rightarrow$''
  It is sufficient to prove that any linear combination of the form $\Gamma_p^1 - \Gamma_p^2 - (\Gamma_k^2 - \Gamma_k^1)$ with $1 \leq k \leq p-2$, lies in the image of $d^1$. 
  Note that $\Gamma_k^2 - \Gamma_k^1$ is the result of performing a STU-relation at the trivalent node adjacent to the leaf with label $k$ in $\Gamma_w = \bar{\psi}_D(w)$, \cf Figure~\ref{fig:gamma12w}.
  Similarly, $\Gamma_p^1 - \Gamma_p^2$ is the result of performing a STU-relation at the trivalent node adjacent to the leaf with label $p-1$ in $\Gamma_w$.
  Therefore, we can obtain the linear combination $\Gamma_p^1 - \Gamma_p^2 - \Gamma_k^2 + \Gamma_k^1$ via performing STU-relations on a $(k,p-1)$-marked unitrivalent graphs at its two marked trivalent nodes.
  By Proposition~\ref{prop:ep-1pcomb}, the domain of the differential $d^1$ is exactly the set of $(k,p-1)$-marked unitrivalent graphs.
\end{proof}

\begin{remark}
  Our proofs of Theorem~\ref{thm:d1comb} and Corollary~\ref{cor:d1imgcomb} supplement the proof of \cite[Proposition~4.8]{Con08} with more details.
  With the notion of marked unitrivalent graph, we are able to give the combinatorial interpretation of the map $d^1$, instead of only its image.
\end{remark}

\begin{corollary}\label{cor:epp2}
  \begin{enumerate}
    \item For $p \geq 4$, the group $E_{p,p}^2$ is isomorphic to the abelian group generated by unitrivalent trees of degree $p-1$, modulo AS-, IHX-, and $\mathnormal{STU}^2$-relations.
    \item For $p = 3$, we have $E_{3,3}^2 \cong \Tca_{2} \cong \ZZ$.
    \item For $p = 0, 1, 2$, we have $E_{p,p}^2 = 0$.
  \end{enumerate}  
\end{corollary}

\begin{proof}
   \rom{1} follows from Lemma \ref{lem:eppalter} and Corollary \ref{cor:d1imgcomb}.
   \rom{2} follows from the fact that the map $d^1 \colon E_{2,3}^1 \to E_{3,3}^1$ is trivial (Proposition~\ref{prop:d1lowdim}) and $E_{3,3}^1 \cong \Tca_{2}$ (Proposition~\ref{prop:eppcomb})
   \rom{3} follows from Proposition~\ref{prop:d1lowdim}.
\end{proof}

\begin{theorem}[Conant]
  Let $n \in \NN$.
  We have $E_{n, n}^{2} \otimes \QQ$ is isomorphic to the rational vector space $\Aca_{n-1}^{I} \otimes \QQ$ generated by unitrivalent trees of degree $n-1$ modulo AS-, IHX-, and SEP-relations. 
\end{theorem}

\begin{proof}
  See \cite[Theorem 3.3]{Con08}.
\end{proof}

\begin{remark}
  We refer the readers to \cite{Bar95}, \cite{CT04}, \cite{CT04b} and \cite{Hab00},  for the close relation between Vassiliev knot invariants and unitrivalent graphs.
  As already points out by Bott in \cite{Bo95}, studying the groups $E_{n, n}^{r}$, and especially the passage from $E_{n,n}^2$ to $E_{n, n}^{\infty}$ may be another approach to the theory of Vassiliev invariants.
\end{remark}


\section{Conclusion and further work}\label{sec:furtherwork}

The connection between Vassiliev invariant and the embedding functor $\Emb(\blank)$ can be expressed in the following diagram with $n \geq 3$
\begin{equation}\label{dig:dream}
  \begin{tikzcd}
    & \pi_0(\pT_n \Emb(\I)) & E_{n,n}^{\infty} \arrow[rdd, phantom, "\star"] \arrow[l, hook] & E_{n,n}^2\arrow[l, two heads] & E_{n,n}^1 \arrow[l, two heads] \\
    \pi_0(\Kca) \arrow[ru, "\eta_n(\I)"] \arrow[rd] &  &  &  &  \\
    & \quot{\pi_0(\Kca)}{\sim_{C_n}} \arrow[uu, "\bar{\eta}_n(\I)"] & \Gca_{n-1} \arrow[l, hook] \arrow[uu, "\restr{\bar{\eta}_n(\I)}{\Gca_{n-1}}"] & \frac{\ZZ[\textnormal{UTT}_{n-1}]}{\textnormal{AS, IHX, } \textnormal{STU}^2} \arrow[l, "\textnormal{\cite{CT04}}", two heads] \arrow[uu, "\cong"] & \frac{\ZZ[\textnormal{UTT}_{n-1}]}{\textnormal{AS, IHX}}.\arrow[l, two heads] \arrow[uu, "\cong"]
  \end{tikzcd}
\end{equation}
Let us first explain and give references to some of the notations in the diagram:
\begin{enumerate}
	\item Recall that $E_{n, n}^{i}$ with $i \in \{1, 2, \infty\}$ denotes the diagonal terms of the homotopy Bousfield--Kan spectral sequence associated to the tower of fibrations
		\[
		\cdots \to \pT_n \Emb(\I) \to \pT_{n-1} \Emb(\I) \to \cdots \to \pT_0 \Emb(\I)
		\]
		with
		\[
		E_{n,n}^{\infty} \cong \ker\big( \pi_0(\pT_n \Emb(\I)) \to \pi_0(\pT_{n-1} \Emb(\I)) \big)
		\]
	\item The clasper surgery equivalence $\sim_{C_n}$ is an equivalence relation on the set $\pi_0(\Kca)$ of isotopy classes of knots defined by Habiro in \cite{Hab00}, and he proves that the canonical map $\pi_0(\Kca) \to \quot{\pi_0(\Kca)}{\sim_{C_n}}$ is the universal additive Vassiliev invariant of degree $n-1$.
	\item The group $\Gca_{n-1}$ has a geometric interpretation using gropes. 
		There is an isomorphism 
		\[
		\Gca_{n-1} \cong \ker \big( \quot{\pi_0(\Kca)}{\sim_{C_n}} \to \quot{\pi_0(\Kca)}{\sim_{C_{n-1}}} \big)
		\]
	    A grope is an embedded (CW)-complex in $\RR^3$, whose boundary components are knots.
	    Conant and Teichner \cite{CT04, CT04b} explain the connection among gropes, clasper surgery and Vassiliev invariants. 
\end{enumerate}

In this text we showed the commutativity of the right most square in the above diagram, by giving combinatorial interpretations to our computations of the spectral sequence.
The key ingredient for the calculations is the equivalence between the category of augmented cosimplicial spaces and the category of good functors from $\bOpen{\I}{}^{\op}$ to spaces, which provides us with a cosimplicial space $\Embcos^{\bullet}$ associated to the embedding functor $\Emb(\blank)$.

We find it interesting that the upper row of Diagram~\ref{dig:dream} is algebraic in nature whereas the lower row is geometric and combinatorial.
If Conjecture~\ref{conj:BCKS} holds, the map $\bar{\eta}_n(\I)$ would be an isomorphism of groups. 
In \cite{Kos20} Kosanovic shows that the square marked with ``$\star$'' commutes, thus making Diagram \ref{dig:dream} commutative.
A natural question that arises for us now is whether we can give a geometric or combinatorial interpretation of the Goodwillie--Weiss tower of knots, \eg using gropes or clasper surgeries.
In our forthcoming work \cite{KST}, we are working towards extending the map $\bar{\eta}_n(\I)$ to a map of spaces.
More specifically, we are constructing a space of gropes together with a continuous map to~$\pT_n\Emb(I)$ such that the induced map on path-connected components is $\bar{\eta}_n(\I)$, for every $n \in \NN$ and $n \geq 0$.

\printbibliography

\end{document}